\documentclass[11pt]{amsart}

\usepackage{amsmath}
\usepackage{amssymb}
\usepackage{latexsym}
\usepackage{amscd}
\usepackage{mathrsfs}
\usepackage[all]{xy}

\newdimen\AAdi%
\newbox\AAbo%
\def\AAk#1#2{\s_etbox\AAbo=\hbox{#2}\AAdi=\wd\AAbo\kern#1\AAdi{}}%
\def\AAr#1#2#3{\s_etbox\AAbo=\hbox{#2}\AAdi=\ht\AAbo\raise#1\AAdi\hbox{#3}}%
\font\tenmsb=msbm10 at 12pt \font\sevenmsb=msbm7 at 8pt
\font\fivemsb=msbm5 at 6pt
\newfam\msbfam
\textfont\msbfam=\tenmsb \scriptfont\msbfam=\sevenmsb
\scriptscriptfont\msbfam=\fivemsb

\newtheorem{theorem}{Theorem}

\newtheorem{remark}[theorem]{Remark}

\newtheorem{corollary}[theorem]{Corollary}

\newtheorem{lemma}[theorem]{Lemma}
\newtheorem{proposition}[theorem]{Proposition}

\numberwithin{equation}{section} \numberwithin{theorem}{section}

\renewcommand{\topmargin}{0cm}
\renewcommand{\oddsidemargin}{5mm}
\renewcommand{\evensidemargin}{5mm}
\renewcommand{\textwidth}{150mm}
\renewcommand{\textheight}{230mm}

\def\R{\mathbb R}

\def\N{\mathbb N}
\def\S{\mathbb S}

\def\mR{\mathcal R}

\def\na{\nabla}
\def\bn{\overline\nabla}

\def\ir#1{\mathbb R^{#1}}

\def\f#1#2{\frac{#1}{#2}}

\def\a{\alpha}
\def\be{\beta}

\def\r{\Re_{I\!V}}

\def\p#1{\partial #1}

\def\de{\delta}
\def\De{\Delta}
\def\e{\eta}
\def\ep{\epsilon}

\def\G{\Gamma}
\def\g{\gamma}
\def\k{\kappa}
\def\la{\lambda}
\def\La{\Lambda}
\def\lan{\langle}
\def\ran{\rangle}

\def\Om{\Omega}
\def\th{\theta}

\def\si{\sigma}
\def\Si{\Sigma}

\def\r{\rho}
\def\z{\zeta}
\def\div{\mathrm{div}}

\begin{document}

\title[Existence and non-existence of area-minimizing hypersurfaces]
{Existence and non-existence of area-minimizing hypersurfaces in
  manifolds of non-negative Ricci curvature}

\author{Qi Ding}
\address{Max Planck Institute for Mathematics in the Sciences, Inselstr. 22, 04103 Leipzig, Germany}
\email{dingqi09@fudan.edu.cn}\email{dingqi@fudan.edu.cn}
\address{Current Address: Shanghai Center for Mathematical Sciences, Fudan University, Shanghai 200433, China}
\author{J. Jost}
\address{Max Planck Institute for Mathematics in the Sciences, Inselstr. 22, 04103 Leipzig, Germany}
\email{jost@mis.mpg.de}
\author{Y.L. Xin}
\address{Institute of Mathematics, Fudan University, Shanghai 200433, China}
\email{ylxin@fudan.edu.cn}

\thanks{The first author is supported partially by Natural Science Foundation of Shanghai (Grant No. 15ZR1402200). The second author is supported by the ERC Advanced Grant
  FP7-267087. The third author is supported partially by NSFC. He is also grateful to the Max Planck
Institute for Mathematics in the Sciences in Leipzig for its
hospitality and continuous support.}

\begin{abstract}
We study minimal hypersurfaces in  manifolds of non-negative Ricci
curvature, Euclidean volume growth and quadratic curvature decay at
infinity. By comparison with capped spherical cones, we identify a
precise borderline for the Ricci curvature decay. Above this value, no
complete area-minimizing hypersurfaces exist. Below this value, in
contrast, we construct  examples.
\end{abstract}

\maketitle

\section{Introduction}
 Bernstein's theorem  says that  an entire
minimal  graph in $\ir{3}$ has to be a plane. This is a classical theorem, and several proofs have been found for it. The original proofs were
strictly two-dimensional, making essential use of conformal
coordinates, but the statement itself is certainly meaningful in any
dimension. Therefore, it was asked whether it also holds in higher dimensions.
 By using and developing tools from geometric measure theory, higher
dimensional generalizations  of the Bernstein theorem were achieved by successive efforts   of
W. Fleming \cite{F}, E. De Giorgi \cite{DG}, F. J. Almgren \cite{Al}
and J. Simons \cite{Si}  up to dimension seven within the framework of
geometric measure theory. In 1969, Bombieri-De Giorgi-Giusti
\cite{BDG} then provided a counterexample by constructing a nontrivial entire minimal graph in
$\ir{n+1}$ with $n>7$ whose tangent cone at infinity had been
described earlier by Simons.

Clearly, the Bernstein problem can be further generalized. We can not
only increase the dimension of the ambient space, but also allow for
more general Riemannian geometries than the Euclidean one. In order to
see what might happen then, we observe that minimal graphs in
Euclidean space are
automatically area minimizing. Thus, the Bernstein problem is
essentially about the (non-)existence of a particular class of complete area-minimizing
hypersurfaces. Therefore, the challenge of the Bernstein problem
consists in finding sharp conditions for the existence or
non-existence of complete  area-minimizing hypersurfaces in curved
ambient manifolds.

Let us therefore review the previous results in this direction. Schoen-Simon-Yau \cite{SSY} obtained
$L^p-$estimates for the squared norm of the second fundamental form for stable minimal hypersurfaces in certain curved ambient manifolds. As a consequence, they showed that any stable minimal hypersurface with Euclidean volume growth in a flat $N^{n+1}$ with $n\le 5$ has to be totally geodesic. Later,
Fischer-Colbrie and Schoen \cite{FS} proved that there are no stable minimal surfaces in 3-dimensional manifolds with positive Ricci curvature. Shen-Zhu \cite{SZ} proved certain rigidity results for stable minimal hypersurfaces in $N^4$ or $N^5$. On the other hand, P. Nabonnand \cite{N} constructed  a complete manifold $N^{n+1}$ with positive Ricci curvature which admits area-minimizing hypersurfaces.
M. Anderson \cite{An} proved a non-existence result for
area-minimizing hypersurfaces in  complete non-compact simply
connected manifolds $N^{n+1}$ of non-negative sectional curvature with
diameter growth conditions. For rotationally symmetric spaces with
conical singularities, some explicit results were obtained by F. Morgan
in \cite{Mf}. These results will provide us with important model
spaces for the general theory.

In the present paper we will study minimal hypersurfaces in
 complete Riemannian manifolds that  satisfy three  conditions:

\noindent C1) non-negative Ricci curvature;\\
\noindent C2) Euclidean volume growth;\\
\noindent C3) quadratic decay of the curvature tensor.

Such manifolds can be much more complicated than Euclidean space, but
on the other hand, this class of manifolds possesses certain
topological and analytical properties \cite{P},\cite{CM1} that
constrain their geometry. They
admit tangent cones at infinity over a smooth compact manifold in the
Gromov-Hausdorff sense. These cones may be not unique, but they
have certain nice properties, proved by Cheeger-Colding \cite{ChC}. Another important
fact is that their Green functions have a well controlled asymptotic
behavior.  In particular,  the Hessian of such a Green function
converges to the  metric tensor (up to a constant factor 2)
point-wisely at infinity, as shown by Colding-Minicozzi \cite{CM1}.
The precise results will be described in section 4.


While our non-existence results are quite general, the existence
results that we develop here, mainly for the purpose of showing that
our non-existence results are sharp, are more explicit and depend on
special constructions. Essentially, for these constructions, we
consider ambient manifolds of the form  $\Si\times\R$   where $\Si$ is
an $n$-dimensional Riemannian manifold with a conformally flat
 metric whose conformal factor depends only on the radius. This class
 will include  a capped spherical cone with opening angle $2\pi \kappa
$,
 denoted by $MCS_\k$. Its  tangent  cone at infinity is the uncapped
 spherical cone $CS_\k$, or equivalently, the Euclidean cone over a
 sphere of radius $\kappa$. These cones  will be on one hand our main
 examples for existence results and on the other hand our model spaces
 for the non-existence results. The border between those two
 phenomena, existence vs. non-existence, will be sharp. Existence takes
 place for $\k \ge \f{2}{n}\sqrt{n-1}$, non-existence else. The
 intuitive geometric reason is simply that for larger values of $\k$,
 in order to minimize area, it is most efficient to go through the
 vertex of the cone, whereas for smaller values of $\k$, it is better
 to avoid the vertex and go around the cone. This had already been
 observed by F. Morgan
in \cite{Mf}. As a by-product we can answer some questions raised by
M. Anderson in \cite{An}.

Whereas the existence examples are specific, our non-existence results
will be general. Essentially, the idea consists in reducing them to
the model cases by taking cones at infinity. For this, we need some
heavier machinery, including the theory of Gromov-Hausdorff limits
\cite{GLP,JK,Ps,GW} and
the theory of currents
in metric spaces developed by  Ambrosio-Kirchheim \cite{AK}. In order
to apply those tools, we shall analyze the Green function at infinity
of the ambient space
and minimal hypersurfaces with Euclidean volume growth, in order to carry the  stability inequality for minimal
hypersurfaces over to the asymptotic limit. The
corresponding results may be of interest in themselves, see Theorem \ref{Volest}.

Our main results thus are general non-existence results for stable minimal hypersurfaces in $(n+1)-$manifolds $N$ with conditions C1), C2) and C3)
under an additional growth condition on the non-radial Ricci curvature
involving a constant $\k'$. For the capped spherical cones $MCS_\k$,
this constant $\k'$ can be
expressed in terms of the constant $\k$.  More precisely, we show
that $N$ admits no complete stable minimal hypersurface with at most Euclidean
volume growth  if the above constant $\k'>\f{(n-2)^2}{4}$, see
Theorem \ref{N-M-E}. The existence result of Theorem \ref{Existence}
then tells us that  our condition on the asymptotic non-radial Ricci
curvature is optimal.

\emph{Acknowledgments.}
The authors would like to thank referees for insightful comments which improved the paper.

\section{Preliminaries}

Let $\Si$ be an $n$-dimensional Riemannian manifold with  metric  $ds^2=\si_{ij}dx_idx_j$ in  local coordinates.
Let $D$ be the corresponding Levi-Civita connection on $\Si$.
For a subset $\Om\subset\Si$ let $M$ be a graph in the product manifold $\Om\times\R$ with smooth defining function $u$ on $\Si,$ i.e.,
\begin{equation}\aligned\label{MOm}
M=\{(x,u(x))\in \Om\times\R|\ x\in\Om\}.
\endaligned
\end{equation}
Since $N=\Si\times\R$ has the product metric $ds^2=\si_{ij}dx_idx_j+dt^2$, then the induced metric  on $M$ is
$$ds^2=g_{ij}dx_idx_j=(\si_{ij}+u_iu_j)dx_idx_j,$$
where $u_i=\f{\p u}{\p x_j}$ and $u_{ij}=\f{\p^2u}{\p x_i\p x_j}$ in the sequel. Let $(\si^{ij})$ be the inverse  metric tensor on $\Si$.
Let $E_i$ and $E_{n+1}$ be the dual vectors of $dx_i$ and $dt$, respectively. Let $\G^k_{ij}$ be the Christoffel symbols of $\Si$ with respect to the frame $E_i$, i.e., $D_{E_i}E_j=\sum_k\G^k_{ij}E_k.$
Set $u^i=\si^{ij}u_j,\; |Du|^2=\si^{ij}u_iu_j,\;
D_iD_ju=u_{ij}-\G^k_{ij}u_k$ and $v=\sqrt{1+|Du|^2}$.  If $f$ stands
for the immersion \eqref{MOm} of $\Si$ in $M\subset N$, then $X_i=f_*E_i=E_i+u_iE_{n+1}, \; i=1, \cdots, n,$ are
tangent vectors of $M$ in $N$. Let $\nu_M$ and $H$ be the unit normal vector field and the mean curvature of $M$ in $N$. Then, direct computation yields $$\nu_M=\f{1}{v}(-\si^{ij}u_jE_i+E_{n+1}),$$
$$H=\div_\Si\left(\f{Du}{v}\right)=\f1{\sqrt{\det \si_{kl}}}\p_j\left(\sqrt{\det \si_{kl}}\f{\si^{ij}u_i}{v}\right).$$

$M$ is a minimal graph in $\Om\times\R$ if and only if  $H\equiv 0$ and $u$ satisfies
\begin{equation}\aligned\label{u}
\div_\Si\left(\f{Du}{\sqrt{1+|Du|^2}}\right)=\f1{\sqrt{\det \si_{kl}}}\p_j\left(\sqrt{\det \si_{kl}}\f{\si^{ij}u_i}{\sqrt{1+|Du|^2}}\right)=0.
\endaligned
\end{equation}
This is the Euler-Lagrangian equation of the volume functional of $M$ in $N$. Moreover, similar to the Euclidean case \cite{X}, any minimal graph on $\Om$ is also an area-minimizing hypersurface in $\Om\times\R$, see Lemma \ref{comp} below.

We introduce an operator $\mathfrak{L}$ on a domain $\Om\subset\Si$ by
\begin{equation}\aligned\label{WL}
\mathfrak{L}F=\left(1+|DF|^2\right)^{\f32}\div_{\Si}\left(\f{D F}{\sqrt{1+|D F|^2}}\right)=\left(1+|D F|^2\right)\De_{\Si}F-F_{i,j}F^iF^j,
\endaligned
\end{equation}
where $F^i=\si^{ik}F_k$, and $F_{i,j}=F_{ij}-\G_{ij}^kF_k$ is the  covariant derivative. Clearly, $\{(x,F(x))|\ x\in\Om\}$ is a minimal graph on $\Si$ if and only if $\mathfrak{L}F=0$ on $\Om$. We call  $F$  $\mathfrak{L}$-\emph{subharmonic} ($\mathfrak{L}$-\emph{superharmonic}) if $\mathfrak{L}F\ge0$ ($\mathfrak{L}F\le0$).

\begin{lemma}\label{comp}
Let $\Om$ be a bounded domain in $\Si$ and $M$ be a minimal graph on
$\overline{\Om}$ as in \eqref{MOm} with volume element $d\mu_M$. For any hypersurface $W\subset\overline{\Om}\times\R$ with $\p M=\p W$, one has
\begin{equation}\aligned
\int_{M}d\mu_M\le\int_{W}d\mu_{W},
\endaligned
\end{equation}
with equality if and only if $W=M$.
\end{lemma}
\begin{proof}
Let $U$ be the domain in $N$ enclosed by $M$ and $W$. Recall that $\nu_M$ is a unit normal vector field on $M$.
Viewing $u_i$ and $v$ as functions on $\Si$, we define a vector field $Y$ such that for every $(x,t)\in U$, $Y$ is just $\nu_M$ at $M\cap(\{x\}\times\R)$ up to a translation along the $E_{n+1}$ axis. Namely,
$$Y(x,t)=-\sum_{i=1}^n\f{\si^{ij}(x)u_j(x)}{v(x)}E_i(x)+\f{1}{v(x)}E_{n+1}.$$
From the minimal surface equation (\ref{u}) we have
\begin{equation*}\aligned
\overline{\div}(Y)=-\sum_i\f1{\sqrt{\det \si_{kl}}}\p_{x_i}\left(\f{\sqrt{\det \si_{kl}}\si^{ij}u_j}{v}\right)=0,
\endaligned
\end{equation*}
where $\overline{\div}$ stands for the divergence operator on $N$.
Let $\nu_M,\nu_W$ be the unit outside normal vectors of $M,W$ respectively. Observe that $Y|_M=\nu_M$. Then by Green's formula,

\begin{equation*}\aligned\label{Gaussf}
0=&\int_{U}\overline{\div}(Y)=\int_{M}\lan Y,\nu_M\ran d\mu_M-\int_{W}\lan Y,\nu_W\ran d\mu_{W}\\
\ge&\int_{M} d\mu_M-\int_{W} d\mu_{W}.
\endaligned
\end{equation*}
Obviously,  equality holds if and only if $M=W$.
\end{proof}

The index form from the second variational formula for the volume
functional for a two-sided minimal hypersurface $M$ in $N$ is (see Chapter 6 of \cite{X})
\begin{equation}\label{SV}
I(\phi,\phi)=\int_M\left(|\na\phi|^2-|\bar{A}|^2\phi^2-Ric_N(\nu_M,\nu_M)\phi^2\right)d\mu_M,
\end{equation}
for any $\phi\in C_c^2(N)$, where $\na$ and $\bar{A}$ are the Levi-Civita connection and the second fundamental form of $M$, respectively.

Let $S_{\k}$ be an $n-$sphere in $\R^{n+1}$ with radius $0<\k\le1$, namely,
$$S_{\k}=\{(x_1,\cdots,x_{n+1})\in\R^{n+1}|\ x_1^2+\cdots+x_{n+1}^2=\k^2\}.$$
If $\{\th_i\}_{i=1}^{n}$ is an orthonormal basis of $S_{\k}$,
then the sectional curvature of $S_{\k}$ is
$$K_S(\th_i,\th_j)=\f{1}{\k^2}\qquad \mathrm{for}\ i\neq j.$$
Let $CS_\k=\R^+\times_\r S_\k$ be the cone over $S_{\k}$ with vertex $o$, which has the metric
$$\si_C=d\r^2+\k^2\r^2d\th^2,$$
where $d\th^2$ is the standard metric on $\S^{n}(1)$.

Let $\{e_\a\}_{\a=1}^{n}\bigcup\{\f{\p}{\p\r}\}$ be an orthonormal basis at the considered point of $CS_\k$ away from the vertex, then the sectional curvature and Ricci curvature of $CS_\k$ are
\begin{equation}\aligned\label{SRCS}
K_{CS_\k}\left(\f{\p}{\p\r},e_\a\right)=0,\quad &K_{CS_\k}(e_\a,e_\be)=\f{1}{\r^2}\left(\f1{\k^2}-1\right),\\
Ric_{CS_\k}\left(\f{\p}{\p\r},\f{\p}{\p\r}\right)=Ric_{CS_\k}\left(\f{\p}{\p\r},e_\a\right)=0,\quad &Ric_{CS_\k}(e_\a,e_\be)=\f{n-1}{\r^2}\left(\f1{\k^2}-1\right)\de_{\a\be}.
\endaligned
\end{equation}

Set $\r=r^\k$, then $\si_C$ can be rewritten as a conformally flat metric
\begin{equation}\aligned\label{siC}
\si_C=\k^2r^{2\k-2}dr^2+\k^2r^{2\k}d\th^2=\k^2r^{2\k-2}\sum_{i=1}^{n+1}dx_i^2=e^{2\log\k-2(1-\k)\log r}\sum_{i=1}^{n+1}dx_i^2,
\endaligned
\end{equation}
where $r^2=\sum_ix_i^2$.

Let $Y$ be an $(n-1)-$dimensional minimal hypersurface in $S_{\k}$ with the second fundamental form $A$ and $CY$ be the cone over $Y$ in $CS_{\k}$ with vertex $o$. For any $0<\ep<1$ denote
$$CY_\ep=\{tx\in S_\k\times\R|\ x\in Y,\ t\in[\ep,1]\}.$$
Clearly, $Y$ is a minimal hypersurface in $S_{\k}$ if and only if $CY_\ep$ is minimal in $CS_{\k}$. Moreover, let $\bar{A}$ be the second fundamental form of $CY_\ep$ in $CS_{\k}$, then
$$|\bar{A}|^2=\f1{\r^2}|A|^2.$$
At any considered point, we can suppose that $\th_{n}$ is the unit
normal vector of $Y\subset S_\k$ and $\{\th_i\}_{i=1}^{n-1}$ is the
orthonormal basis of $TY$. Let $\nu=\f1\r\th_{n}$ be the unit normal
vector of $CY_{\ep}$. Let $d\mu$ and $d\mu_Y$ be the volume element of
$CY_\ep$ and $Y$, respectively. 

Now, from (\ref{SV}), the index form of $CY_{\ep}$ in $CS_{\k}$ becomes
\begin{equation}\aligned
I(\phi,\phi)=\int_{CY_\ep}\left(-\phi\De_{CY}\phi-|\bar{A}|^2\phi^2-Ric_{CS_{\k}\times\R}(\nu,\nu)\phi^2\right)d\mu
\endaligned
\end{equation}
for any $\phi\in C_c^2(CY\setminus\{o\})$. Note $Ric_{S_\k}(\th_i,\th_j)=\f{n-1}{\k^2}\de_{ij}$ and
$$Ric_{CS_k}(\nu,\nu)=\f1{\r^2}Ric_{S_k}(\th_{n},\th_{n})-\f{n-1}{\r^2}=\f{n-1}{\r^2}\left(\f1{\k^2}-1\right).$$

When $\phi$ is written as $\phi(x,\r)\in C^2(Y\times\R)$,  a simple calculation implies
\begin{equation}\aligned
\De_{CY} \phi=\f1{\r^2}\De_{Y}\phi+\f{n-1}{\r}\f{\p \phi}{\p\r}+\f{\p^2\phi}{\p \r^2},
\endaligned
\end{equation}
then
\begin{equation}\aligned\label{Indexphi}
I(\phi,\phi)=\int_\ep^1\bigg(\int_Y\bigg(&-\De_Y\phi-|A|^2\phi-\f{n-1}{\k^2}\phi+(n-1)\phi\\
&-(n-1)\r\f{\p\phi}{\p\r}-\r^2\f{\p^2\phi}{\p \r^2}\bigg)\phi\ d\mu_Y\bigg)\r^{n-3}d\r.
\endaligned
\end{equation}

When $\k=1$ and $Y$ is the Clifford minimal hypersurface in the unit $7-$sphere
$$Y=S^3\left(\f{\sqrt{2}}{2}\right)\times S^3\left(\f{\sqrt{2}}{2}\right),$$
then, $CY$ is  Simons' cone, proved to be stable in \cite{Si} (see also Chapter 6 of \cite{X}).

\section{Constructions of area-minimizing hypersurfaces}

Let $\Si$ be the Euclidean space $\R^{n+1}$ with a conformally flat  metric
$$ds^2=e^{\phi(r)}\sum_{i=1}^{n+1}dx_i^2,$$
where $r=|x|=\sqrt{x_1^2+\cdots+dx_{n+1}^2}$ and $\phi(|x|)$ is smooth in $\R^{n+1}$. Let $F$  be a function on $\R^{n+1}$. Let $E_i=\{\f{\p}{\p x_i}\}$ be a standard basis of $\R^{n+1}$ and $F_i=\f{\p}{\p x_i}F$ be the ordinary derivative in $\R^{n+1}$. Moreover,
$$\G_{ij}^k=\f{\phi'}2\left(\de_{ik}\f{x_j}r+\de_{jk}\f{x_i}r-\de_{ij}\f{x_k}r\right).$$

Denote $|\p F|^2=\sum_iF_i^2$.
Let $\De$ be the standard Laplacian of $\R^{n+1}$, then
\begin{equation}\aligned
\De_\Si F=&\si^{ij}F_{i,j}=e^{-\phi}\de_{ij}\left(F_{ij}-\f{\phi'}2\left(\de_{ik}\f{x_j}r+\de_{jk}\f{x_i}r-\de_{ij}\f{x_k}r\right)F_k\right)\\
=&e^{-\phi}\left(\De F+\f{n-1}2\phi'F_i\f{x_i}r\right).
\endaligned
\end{equation}
By (\ref{WL}) we can compute $\mathfrak{L}F$  in the conformal flat metric as follows.
\begin{equation}\aligned\label{LFx}
\mathfrak{L}F=&e^{-\phi}\left(1+e^{-\phi}|\p F|^2\right)\left(\De F+\f{n-1}2\phi'F_i\f{x_i}r\right)-e^{-2\phi}\left(F_{ij}F_iF_j-\f{|\p F|^2}2\phi'F_i\f{x_i}r\right)\\
=&e^{-\phi}\Big(\left(1+e^{-\phi}|\p F|^2\right)\De F-e^{-\phi}F_{ij}F_iF_j\Big)+e^{-\phi}\left(\f{n-1}2+\f{n}2e^{-\phi}|\p F|^2\right)\phi'F_i\f{x_i}r\\
=&e^{-2\phi}\left(|\p F|^2\left(\De F+\f {n}2\phi'F_i\f{x_i}r\right)-F_{ij}F_iF_j\right)+e^{-\phi}\left(\De F+\f{n-1}2\phi'F_i\f{x_i}r\right).
\endaligned
\end{equation}

\begin{lemma} Let $F=F(\th,r)$ be a function with
\begin{equation}\aligned
\th=\f{x_{n+1}}{\sqrt{x_1^2+\cdots+x_{n+1}^2}},\qquad r=\sqrt{x_1^2+\cdots+x_{n+1}^2},
\endaligned
\end{equation}
on $[-1,1]\times(0,\infty)$. Then we have
\begin{equation}\aligned\label{LFphi}
\mathfrak{L}F=&e^{-2\phi}\bigg(n\left(\left(1-\th^2\right)\f{F_\th^2}{r^2}+F_r^2\right)\left(\f{F_r}r+\f{\phi'}2F_r-\f{\th F_\th}{r^2}\right)\\
&+(1-\th^2)\f{F_\th^2}{r^2}\left(\f{\th F_\th}{r^2}+\f{F_r}r\right)+\f{1-\th^2}{r^2}\left(F_\th^2F_{rr}+F_r^2F_{\th\th}-2F_\th F_rF_{r\th}\right)\bigg)\\
&+e^{-\phi}\left(F_{rr}+\f{1-\th^2}{r^2}F_{\th\th}+\f {n}rF_r-\f{n\th}{r^2}F_\th+\f {n-1}2\phi'F_r\right).
\endaligned
\end{equation}
\end{lemma}
\begin{proof}
For $1\le \a\le n$ we have
\begin{equation}\aligned
F_\a=\p_{x_\a}F=&F_\th\cdot\left(-\f{x_\a x_{n+1}}{r^3}\right)+F_r\f{x_\a}r,\\
F_{n+1}=\p_{x_{n+1}}F=&F_\th\cdot\left(\f1r-\f{x^2_{n+1}}{r^3}\right)+F_r\f{x_{n+1}}r=F_\th\f{\sum_\a x^2_\a}{r^3}+F_r\f{x_{n+1}}r.
\endaligned
\end{equation}
Hence
\begin{equation}\aligned
|\p F|^2=\sum_\a F_\a^2+F^2_{n+1}=F_\th^2\f{\sum_\a x^2_\a}{r^4}+F_r^2=\left(1-\th^2\right)\f{F_\th^2}{r^2}+F_r^2,
\endaligned
\end{equation}
and
\begin{equation}\aligned
\sum_{i=1}^{n+1}x_iF_i=\sum_\a x_\a F_\a+x_{n+1}F_{n+1}=rF_r.
\endaligned
\end{equation}

In polar coordinates,
$$\sum_{i=1}^{n+1}dx_i^2=dr^2+r^2\left(d\be^2+\cos^2\be\ dS^{n-1}\right),$$
where $\sin\be=\th\in[-1,1]$ and $dS^{n-1}$ is the standard metric in the unit sphere $\S^{n-1}\in\R^n$. Hence
$$\sum_{i=1}^{n+1}dx_i^2=dr^2+\f{r^2}{1-\th^2}d\th^2+r^2(1-\th^2)dS^{n-1},$$
and
\begin{equation}\aligned
\De F=&\f1{r^n(1-\th^2)^{\f n2-1}}\left(\p_r\left(r^n(1-\th^2)^{\f n2-1}F_r\right)+\p_\th\left(r^n(1-\th^2)^{\f n2-1}\f{1-\th^2}{r^2}F_\th\right)\right)\\
=&F_{rr}+\f nrF_r+\f{1-\th^2}{r^2}F_{\th\th}-\f{n\th}{r^2}F_\th.
\endaligned
\end{equation}
Moreover,
\begin{equation}\aligned
&\sum_{1\le i,j\le n+1}F_{ij}F_iF_j=\f12\sum_iF_i\p_i|\p F|^2\\
=&\f12\sum_\a\left(-\f{x_\a x_{n+1}}{r^3}F_\th+\f{x_\a}r F_r\right)\left(-\f{x_\a x_{n+1}}{r^3}\p_\th|\p F|^2+\f{x_\a}r \p_r|\p F|^2\right)\\
&+\f12\left(\f{\sum_\a x^2_\a}{r^3}F_\th+\f{x_{n+1}}r F_r\right)\left(\f{\sum_\a x^2_\a}{r^3}\p_\th|\p F|^2+\f{x_{n+1}}r \p_r|\p F|^2\right)\\
=&\f12\f{\sum_\a x_\a^2}{r^4}F_\th\p_\th |\p F|^2+\f12F_r\p_r|\p F|^2\\
=&\f{1-\th^2}{2r^2}F_\th\p_\th\left(\left(1-\th^2\right)\f{F_\th^2}{r^2}+F_r^2\right)+\f12F_r\p_r\left(\left(1-\th^2\right)\f{F_\th^2}{r^2}+F_r^2\right)\\
=&-\th(1-\th^2)\f{F_\th^3}{r^4}+(1-\th^2)^2\f{F_\th^2 F_{\th\th}}{r^4}+2(1-\th^2)\f{F_\th F_rF_{r\th}}{r^2}\\
&-(1-\th^2)\f{F_\th^2F_r}{r^3}+F_r^2F_{rr}.
\endaligned
\end{equation}
Hence by \eqref{LFx} we have
\begin{equation}\aligned
\mathfrak{L}F=&e^{-2\phi}\bigg(\left(\left(1-\th^2\right)\f{F_\th^2}{r^2}+F_r^2\right)\left(F_{rr}+\f nrF_r+\f{1-\th^2}{r^2}F_{\th\th}-\f{n\th}{r^2}F_\th+\f n2\phi'F_r\right)\\
&-\bigg(-\th(1-\th^2)\f{F_\th^3}{r^4}+(1-\th^2)^2\f{F_\th^2 F_{\th\th}}{r^4}+2(1-\th^2)\f{F_\th F_rF_{r\th}}{r^2}-(1-\th^2)\f{F_\th^2F_r}{r^3}\\
&+F_r^2F_{rr}\bigg)\bigg)+e^{-\phi}\left(F_{rr}+\f nrF_r+\f{1-\th^2}{r^2}F_{\th\th}-\f{n\th}{r^2}F_\th+\f {n-1}2\phi'F_r\right)\\
=&e^{-2\phi}\bigg(n\left(\left(1-\th^2\right)\f{F_\th^2}{r^2}+F_r^2\right)\left(\f{F_r}r+\f{\phi'}2F_r-\f{\th F_\th}{r^2}\right)\\
&+(1-\th^2)\f{F_\th^2}{r^2}\left(\f{\th F_\th}{r^2}+\f{F_r}r\right)+\f{1-\th^2}{r^2}\left(F_\th^2F_{rr}+F_r^2F_{\th\th}-2F_\th F_rF_{r\th}\right)\bigg)\\
&+e^{-\phi}\left(F_{rr}+\f{1-\th^2}{r^2}F_{\th\th}+\f nrF_r-\f{n\th}{r^2}F_\th+\f {n-1}2\phi'F_r\right).
\endaligned
\end{equation}
\end{proof}

\begin{theorem} \label{LF}
Let $\Si$ be an $(n+1)-$dimensional Euclidean space $\ir{n+1},\, n\ge
2,$ endowed with a smooth conformally flat metric $ds^2=e^\phi\sum dx_i^2,$ where  
$\phi'(r)\ge-2(1-\k)r^{-1}$ and $\f{2}{n}\sqrt{n-1}\le\k\le1$. If
$$F(\th,r)=C\th r^p=Cx_{n+1}r^{p-1}\triangleq \mathcal{F}(x_{n+1},r)$$
with any constant $C>0$ and
$p=\f{n}2\k-\sqrt{\f{n^2\k^2}4-(n-1)},$
then except at the origin we have
\begin{equation}\aligned\label{LFxnr}
\mathfrak{L}\mathcal{F}(x_{n+1},r)\left\{\begin{array}{cc}
     \ge0     & \quad\ \ \ {\rm{if}} \ \ \  (x_1,\cdots,x_n)\in\R^n,\ x_{n+1}\ge 0 \\ [3mm]
     \le0     & \quad\ \ \ {\rm{if}} \ \ \  (x_1,\cdots,x_n)\in\R^n,\ x_{n+1}\le 0
     \end{array}\right..
\endaligned
\end{equation}
\end{theorem}
\begin{proof}
Since $\phi'\ge-2(1-\k)r^{-1}$ for $0<\k\le1$ and $F_r=Cp\th r^{p-1}$. By \eqref{LFphi} except at the origin  we have
\begin{equation}\aligned\label{LFw0}
\th\mathfrak{L}F\ge&\th e^{-2\phi}\bigg(n\left(\left(1-\th^2\right)\f{F_\th^2}{r^2}+F_r^2\right)\left(\f{\k F_r}r-\f{\th F_\th}{r^2}\right)\\
&+(1-\th^2)\f{F_\th^2}{r^2}\left(\f{\th F_\th}{r^2}+\f{F_r}r\right)+\f{1-\th^2}{r^2}\left(F_\th^2F_{rr}+F_r^2F_{\th\th}-2F_\th F_rF_{r\th}\right)\bigg)\\
&+\th e^{-\phi}\left(F_{rr}+\f{1-\th^2}{r^2}F_{\th\th}+\big((n-1)\k+1\big)\f{F_r}r-\f{n}{r^2}\th F_\th\right).
\endaligned
\end{equation}

Furthermore, we take the derivatives of $F$ and get
\begin{equation}\aligned
\th\mathfrak{L}F\ge&C^3\th e^{-2\phi}\bigg(n\Big(\left(1-\th^2\right)r^{2p-2}+\th^2p^2r^{2p-2}\Big)\left(\k\th pr^{p-2}-\th r^{p-2}\right)\\
&+(1-\th^2)r^{2p-2}\left(\th r^{p-2}+\th pr^{p-2}\right)+\f{1-\th^2}{r^2}\left(p(p-1)\th r^{3p-2}-2p^2\th r^{3p-2}\right)\bigg)\\
&+C\th e^{-\phi}\Big(p(p-1)\th r^{p-2}+\big((n-1)\k+1\big)p\th r^{p-2}-n\th r^{p-2}\Big)\\
=&C^3\th e^{-2\phi}\Big(\left(n(\k p-1)+1-p^2\right)(1-\th^2)+np^2(\k p-1)\th^2\Big)\th r^{3p-4}\\
&+C\th e^{-\phi}\Big(p^2+(n-1)\k p-n\Big)\th r^{p-2}.
\endaligned
\end{equation}
Note
$$n(\k p-1)+1-p^2=-\left(p-\f{n\k}2\right)^2+\f{n^2\k^2}4-(n-1)=0.$$
By the definition of $p$, we obtain
\begin{equation}\aligned
p=&\f{n\k}2\left(1-\sqrt{1-\f{4(n-1)}{n^2\k^2}}\right)=\f{n\k}2\left(1-\f{n-2}n\sqrt{1-\f{4(n-1)}{(n-2)^2}\left(\f1{\k^2}-1\right)}\right)\\
\ge&\f{n\k}2\left(1-\f{n-2}n\left(1-\f{2(n-1)}{(n-2)^2}\left(\f1{\k^2}-1\right)\right)\right)=\f1\k\left(1+\f{1-\k^2}{n-2}\right)\ge\f1\k.
\endaligned
\end{equation}
Hence
\begin{equation}\aligned
\th\mathfrak{L}F\ge&C^3e^{-2\phi}np^2(\k p-1)\th^4r^{3p-4}+Ce^{-\phi}\Big(p^2+(n-1)\k p-n\Big)\th^2 r^{p-2}\\
\ge&Ce^{-\phi}(p^2-1)\th^2 r^{p-2}\ge0.
\endaligned
\end{equation}
We complete the proof.
\end{proof}

\begin{remark}
There are other $\mathfrak{L}$-sub(super)harmonic functions on $\Si$. For instance, for all $j>0$, $\mathfrak{L}(jx_{n+1}w^{p-1})\ge0$ on $x_{n+1}\ge0$ and $\mathfrak{L}(jx_{n+1}w^{p-1})\le0$ on $x_{n+1}\le0$, where $w=\sqrt{x_1^2+\cdots+x_n^2}$.
\end{remark}

Denote $B_R=\{(x_1,\cdots,x_{n+1})\in\R^{n+1}|\ x_1^2+\cdots+x_{n+1}^2\le R^2\}$.
\begin{theorem}\label{Existence}
If $n\ge3$ and
$$\f{2}{n}\sqrt{n-1}\le\k<1,$$
then any hyperplane through the origin in $\Si$ as described in
Theorem \ref{LF}, that is,  $\ir{n+1}$ equipped with a particular conformally flat metric, is area-minimizing.
\end{theorem}

\begin{proof}
We shall show
that  the hyperplane $T=\{(x_1,\cdots,x_{n+1})\in\R^{n+1}|\
x_{n+1}=0\}$ in $\Si$ with the induced metric is area-minimizing.

Set $\tilde{\phi}(r)=\int_0^re^{\f{\phi(r)}2}dr$. Let us define
$\r=\tilde{\phi}(r)$ and $\la(\r)=r\tilde{\phi}'(r)$, then the
Riemannian metric in $\Si$ can be written in  polar coordinates as
$ds^2=d\r^2+\la^2(\r)d\th^2$, where $d\th^2$ is the  standard metric on $\S^{n}(1)$. Moreover,
\begin{equation}\aligned\label{dladr}
\f{d\la}{d\r}=\f{d\la}{dr}\f{dr}{d\r}=\left(\tilde{\phi}'+r\tilde{\phi}''\right)\f1{\tilde{\phi}'}=1+r(\log\tilde{\phi}')'=1+\f12r\phi'\ge1-(1-\k)=\k.
\endaligned
\end{equation}

When $n\ge 3$ and $q=p-1$ for $p$ as in the statement of Theorem \ref{LF},
let $\mathcal{F}_j(x_{n+1},r)=jx_{n+1}r^q$ for $j>0$ with $r=\sqrt{x_1^2+\cdots+x_{n+1}^2}$. By Theorem \ref{LF} we obtain
\begin{equation}\aligned
\mathfrak{L}\mathcal{F}_j(x_{n+1},r)\left\{\begin{array}{cc}
     \ge0     & \quad\ \ \ {\rm{in}} \ \ \  \{(x_1,\cdots,x_{n+1})\in\R^{n+1}|\ x_{n+1}\ge 0\}\setminus\{0\} \\ [3mm]
     \le0     & \quad\ \ \ {\rm{in}} \ \ \  \{(x_1,\cdots,x_{n+1})\in\R^{n+1}|\ x_{n+1}\le 0\}\setminus\{0\}
     \end{array}\right..
\endaligned
\end{equation}
Combining \eqref{dladr} and formula (2.9) in \cite{D}, we know that any geodesic sphere $\p D_\r\subset\Si$ centered at the origin has positive inward mean curvature $H=(n-1)\f{\la'(\r)}{\la(\r)}>0$ for any $\r>0$.
By the existence theorem for the  Dirichlet problem for minimal hypersurfaces in $\Si\times\R$, see Theorem 1.5 in \cite{Sp},
for any constant $R>0$ and $j=1,2,\cdots,\infty$, there is a solution $u_j\in C^{\infty}(B_{jR})$ to the Dirichlet problem
\begin{equation}\aligned
\left\{\begin{array}{cc}
     \mathfrak{L}u_j=0     & \quad\ \ \ {\rm{in}}  \  B_{jR} \\ [3mm]
     u_j=\mathcal{F}_j    & \quad\ \ \ {\rm{on}}  \  \p B_{jR}
     \end{array}\right..
\endaligned
\end{equation}
By  symmetry,  $u_j=0$ on $B_{R^*}\cap T$ for any fixed $R^*>0$. In fact, by the definition \eqref{LFx} of $\mathfrak{L}$ in a conformally flat metric, $-u_j(x_1,\cdots,x_n,-x_{n+1})$ is also a smooth solution to the above Dirichlet problem. Since Lemma \ref{comp} implies the uniqueness of the minimal graph, we get $u_j(x_1,\cdots,x_n,x_{n+1})=-u_j(x_1,\cdots,x_n,-x_{n+1})$, from which it follows that $u_j=0$ on $B_{R^*}\cap T$.
Let $U=\left\{(x_1,\cdots,x_{n+1})\in\R^{n+1}\big|\  x_{n+1}> 0\right\}$, then the comparison theorem on $B_{R^*}\setminus\{0\}$ implies
\begin{equation}\aligned
\lim_{j\rightarrow\infty}u_j\ge\lim_{j\rightarrow\infty}\mathcal{F}_j=+\infty\qquad \mathrm{in}\ \ B_{R^*}\cap U\\
\endaligned
\end{equation}
and
\begin{equation}\aligned
\lim_{j\rightarrow\infty}u_j\le\lim_{j\rightarrow\infty}\mathcal{F}_j=-\infty\qquad \mathrm{in}\ \ B_{R^*}\cap (\R^{n+1}\setminus \overline{U}).
\endaligned
\end{equation}
Let $U_j$ denote the subgraph of $u_j$ in $B_{R^*}\times\R$, namely,
$$U_j=\left\{(x,t)\in B_{R^*}\times\R\big|\ t<u_j(x)\right\}.$$
Clearly, its characteristic function $\chi_{_{U_j}}$ converges in $L^1_{loc}(B_{R^*}\times\R)$ to $\chi_{_{U\times\R}}$.
By an analogous argument as in Lemma 9.1 in \cite{Gi} for the Euclidean case,
for any compact set $E\subset B_{R^*}\times\R$, that  Graph$(u_j)\triangleq\{(x,u_j(x))|\ x\in\R^{n+1}\}$ is an area-minimizing hypersurface implies that $(U\times\R)\cap E$ is a minimizing set in $E$.
Hence $U\times\R$ is a minimizing set in $B_{R^*}\times\R\subset
\Si\times\R$. By an analogous argument as in Proposition 9.9 in
\cite{Gi} for the Euclidean case, $U$ is a minimizing set in $B_{R^*}$, namely, the hyperplane $T$ minimizes the  perimeter in $B_{R^*}$. Since $R^*$ is arbitrary, we complete the proof.
\end{proof}

As we showed in the previous section, on the cone $CS_\k$ the usual metric can be rewritten as a conformally flat one. We shall therefore  modify the constructions  from the cone $CS_\k$.

\begin{lemma}
Let $\Lambda$ be the rotationally symmetric function on $\R^{n+1}$ defined by
\begin{equation}\aligned\label{La}
\La(x)=\left\{\begin{array}{cc}
      \f{\sqrt{1-\k^2}}{\k}\sqrt{x_1^2+\cdots+x_{n}^2}&      \hskip0.6in \ \ \ {\rm{on}} \ \ \   \R^{n+1}\setminus B_1 \\
      \f{\sqrt{1-\k^2}}{\k}\left(1-\f{2}{\pi}\int^1_{|x|}\left(\arctan \xi(s)\right)ds\right)&      \quad \ \ \ {\rm{on}} \ \ \ B_1
     \end{array}\right.,
\endaligned
\end{equation}where $\xi(s)=s\left(e^{\f1{1-s^2}}-e\right).$
It is a smooth convex function on $\ir{n+1}$.
\end{lemma}
\begin{proof}
In fact, $\xi'(0)=0$, $\xi^{(2k)}(0)=0$ for $k\ge0$ and $\xi^{(j)}(1)=+\infty$ for $j\ge0$. Then on $B_1$
\begin{equation}\aligned
\p_i\Lambda(x)=&\f{2\sqrt{1-\k^2}}{\k\pi}\f{x_i}{|x|}\arctan\xi(|x|),\\
\p_{ij}\Lambda(x)=&\f{2\sqrt{1-\k^2}}{\k\pi}\left(\de_{ij}-\f{x_ix_j}{|x|^2}\right)\f{\arctan\xi}{|x|}+\f{2\sqrt{1-\k^2}}{\k\pi}\f{\xi'}{1+\xi^2}\f{x_ix_j}{|x|^2}.
\endaligned
\end{equation}
Since
$$\f{\arctan\xi(\sqrt{t})}{\sqrt{t}}=\sum_{k=0}^{\infty}\f{(-1)^k}{2k+1}t^k\left(e^{\f1{1-t}}-e\right)^{2k+1}$$
in $[0,\ep]$ for small $\ep>0$,
$t^{-\f12}\arctan\xi(\sqrt{t})$ is a smooth function for $t\in[0,1)$ and
$$\Lambda(x)=\f{\sqrt{1-\k^2}}{\k}\left(1-\f{1}{\pi}\int^1_{|x|^2}\f{\arctan \xi(\sqrt{t})}{\sqrt{t}}dt\right)$$
is a smooth convex function on $B_1$. Denote $\La(r)=\La(|x|)$, then the radial derivative of $\Lambda$ at 1 is
$$\lim_{r\rightarrow1}\p_r\Lambda(r)=\f{2\sqrt{1-\k^2}}{\k\pi}\arctan\xi(1)=\f{\sqrt{1-\k^2}}{\k},$$
and the higher order  radial derivative of $\Lambda$ at 1 is
\begin{equation}\aligned\nonumber
\lim_{r\rightarrow1}(\p_r)^{j+1}\Lambda(r)=&\f{2\sqrt{1-\k^2}}{\k\pi}(\p_r)^j\arctan\xi(r)\Big|_{r=1}\\
=&\f{2\sqrt{1-\k^2}}{\k\pi}(\p_r)^{j-1}\left(\f{\xi'}{1+\xi^2}\right)\bigg|_{r=1}=0\qquad \mathrm{for}\ j\ge1.
\endaligned
\end{equation}
Hence $\Lambda$ is a smooth convex function on $\R^{n+1}$.
\end{proof}

Now we suppose that $MCS_\k$ is an $(n+1)$-dimensional smooth entire graphic hypersurface in $\R^{n+2}$ with the defining function $\Lambda$. We see that
it has non-negative sectional curvature everywhere. In fact, $MCS_\k$
is a $\k-$sphere cone $CS_\k$ with a smooth cap, which we shall call the modified $\k-$ sphere cone.

We already showed that the metric of  the $\k-$sphere cone is
conformally flat,  and we shall now also derive this for $MCS_\k$.

\begin{lemma}\label{CF}
The $(n+1)-$dimensional $MCS_\k$ has a smooth conformally flat metric $$ds^2=e^{\Phi(r)}\sum_{1\le i\le n+1}dx_i^2$$
on $\ir{n+1}$ with $-\f2r(1-\k)\le\Phi'\le0$.
\end{lemma}

\begin{proof} $MCS_\k$ is defined as an entire graph on $\ir{n+2}$. Its induced metric can also be written in
polar coordinates as
\begin{equation}\aligned\label{Sisi}
ds^2=d\r^2+\lambda^2(\r)d\th^2,
\endaligned
\end{equation}
where $d\th^2$ is a standard metric on $\S^{n}(1)$, and
\begin{equation}\aligned\label{lar}
\lambda(\r)=\left\{\begin{array}{cc}
     \k\left(\r+\f1\k-\r_0\right)     & \quad\ \ \ {\rm{for}}  \  \r\ge\r_0 \\ [3mm]
     \z(\r)    & \quad\ \ \ {\rm{for}}  \  0\le\r\le\r_0
     \end{array}\right..
\endaligned
\end{equation}
Here
$$1<\r_0=\int_0^1\sqrt{1+(\p_r\Lambda)^2}dr<\f1\k,$$
and the inverse function of $\z$ satisfies
$$\z^{-1}(s)=\int_0^s\sqrt{1+(\p_r\Lambda)^2}dr,$$
where $\La$ is defined in the last lemma.  Moreover, $\k\le\z'\le1$.

Let $\psi(r)$ be a function on $\Big[0,\left(\f1\k\right)^{\f1\k}\Big)$ with $\psi\left(\left(\f1\k\right)^{\f1\k}\right)=\r_0$ and
\begin{equation}\aligned\label{psi'}
\psi'(r)=\f1r\z(\psi(r))\qquad \mathrm{on} \quad \Big[0,\left(\f1\k\right)^{\f1\k}\Big).
\endaligned
\end{equation}
In fact, let $\tilde{\z}(\r)=\int_1^\r\f1{\z(t)}dt$ for $\r\in(0,\r_0]$, then we integrate the above ordinary differential equation and obtain
$$\tilde{\z}(\psi(r))-\tilde{\z}(\r_0)=\log r+\f1\k\log\k.$$
Since $\tilde{\z}$ is a monotonic function, we can solve this equation for $\psi$.
Since $\k \r\le\z(\r)\le \r$ on $[0,\r_0]$, a standard comparison argument implies that
$$\left(\f1\k\right)^{-\f1\k}\r_0r\le\psi(r)\le\k\r_0r^{\k}\qquad \mathrm{on} \quad \Big[0,\left(\f1\k\right)^{\f1\k}\Big].$$
In particular, $\psi(0)=0$.
Since
$$\psi''(r)=\f{\z'}r\psi'-\f{\z}{r^2}=\f{\z}{r^2}(\z'-1),$$
then,
\begin{equation}\label{cf}
\f{\k-1}r\le\f{\psi''(r)}{\psi'(r)}=\f{\z'-1}r\le0.
\end{equation}
Let
\begin{equation}\aligned\label{ss1}
\r=\tilde{\psi}(r)=\left\{\begin{array}{cc}
     r^\k-\f1\k+\r_0     & \quad\ \ \ {\rm{for}}  \  r\ge\left(\f1\k\right)^{\f1\k} \\ [3mm]
     \psi(r)    & \quad\ \ \ {\rm{for}}  \  0\le r\le\left(\f1\k\right)^{\f1\k}
     \end{array}\right.,
\endaligned
\end{equation}
then $\tilde{\psi}$ also satisfies \eqref{psi'} and hence $\tilde{\psi}$ is smooth on $[0,\infty)$.
Set
\begin{equation}\aligned\label{ss2}
e^{\Phi(r)}=\left(\tilde{\psi}'(r)\right)^2=\left\{\begin{array}{cc}
     \k^2 r^{2\k-2}     & \quad\ \ \ {\rm{for}}  \  r\ge\left(\f1\k\right)^{\f1\k} \\ [3mm]
     (\psi')^2(r)    & \quad\ \ \ {\rm{for}}  \  0\le r\le\left(\f1\k\right)^{\f1\k}
     \end{array}\right.,
\endaligned
\end{equation}
then
\begin{equation}\aligned\label{ss3}
ds^2=e^{\Phi(r)}dr^2+e^{\Phi(r)}r^2d\th^2=e^{\Phi(r)}\sum_{1\le i\le n+1}dx_i^2,
\endaligned
\end{equation}
where $r^2=\sum_ix_i^2$. By (\ref{siC}) and (\ref{cf}) we have  $$-\f2r(1-\k)\le\Phi'\le0.$$
\end{proof}

Now, Lemma \ref{CF} and  Theorem \ref{Existence} yield the following conclusion.
\begin{theorem}\label{MCSK}
Let $n\ge3$. If
$$\f{2}{n}\sqrt{n-1}\le\k<1,$$
then any hyperplane through the origin in $MCS_\k$ is area-minimizing.
\end{theorem}

\begin{remark}\label{MCSk}
Let $\{e_\a\}_{\a=1}^{n}\bigcup\{\f{\p}{\p\r}\}$ be an orthonormal basis at the considered point of $MCS_\k$. Compared with \eqref{SRCS} we calculate the sectional curvature and Ricci curvature of $MCS_\k$ as follows (see Appendix A in \cite{Pl} for instance).
\begin{equation}\aligned\label{secRicSi}
K_{MCS_\k}\left(\f{\p}{\p\r},e_\a\right)=-\f{\la''}{\la}\ge0,\quad &K_{{MCS_\k}}(e_\a,e_\be)=\f{1-(\la')^2}{\la^2}\ge0,\\ Ric_{{MCS_\k}}\left(\f{\p}{\p\r},e_\a\right)=0,\quad
&Ric_{{MCS_\k}}\left(\f{\p}{\p\r},\f{\p}{\p\r}\right)=-n\f{\la''}{\la}\ge0,\\
Ric_{{MCS_\k}}(e_\a,e_\be)=\bigg((n-1)&\f{1-(\la')^2}{\la^2}-\f{\la''}{\la}\bigg)\de_{\a\be}
\ge0.
\endaligned
\end{equation}
In particular, for $\r\ge\r_0$ with $1<\r_0<\f1\k$ we have
\begin{equation}\aligned\label{secRicSi1}
K_{{MCS_\k}}\left(\f{\p}{\p\r},e_\a\right)=0,\quad &K_{{MCS_\k}}(e_\a,e_\be)=\f{1-\k^2}{\k^2(\r+\f1\k-\r_0)^2},\\ Ric_{{MCS_\k}}\left(\f{\p}{\p\r},\f{\p}{\p\r}\right)=&Ric_{{MCS_\k}}\left(\f{\p}{\p\r},e_\a\right)=0,\\
Ric_{{MCS_\k}}(e_\a,e_\be)=(n-1)&\f{1-\k^2}{\k^2(\r+\f1\k-\r_0)^2}\de_{\a\be}.
\endaligned
\end{equation}
From the construction above we see that $MCS_\k$ is a complete simply connected manifold with non-negative sectional curvature.
\end{remark}

\begin{remark}
Since $MCS_\k$ in Theorem \ref{Existence} cannot split off a Euclidean
factor $\ir{}$ isometrically, the Cheeger-Gromoll splitting theorem
\cite{CG} implies that it does not contain a line. Consequently, this gives a negative  answer to the question (1) in \cite{An}, which is

If $M$ is a complete area-minimizing hypersurface in a complete simply connected manifold $N$ of non-negative curvature, does it follow that $N$ contains a line, that is a complete length-minimizing geodesic?
\end{remark}

If we define for each $x\in\R^n$
\begin{equation}\aligned\nonumber
\widetilde{\Lambda}(x)=\f{2\sqrt{1-\k^2}}{\pi\k}\int^{|x|}_0\arctan sds,
\endaligned
\end{equation}
then $\widetilde{\La}$ is a smooth strictly convex function on $\R^n$ and the hypersurface $\widetilde{\Si}=\{(x,\widetilde{\La}(x))|\ x\in\R^n\}$ is a smooth manifold with positive sectional curvature everywhere. In fact, $\widetilde{\Si}$ can be seen as a Riemannian manifold $(\R^n,\tilde{\si})$ with
$$\tilde{\si}=d\r^2+\tilde{\la}^2(\r)d\th^2$$
in polar coordinates, where the inverse function of $\tilde{\la}$ satisfies
$$\tilde{\la}^{-1}(s)=\int_0^s\sqrt{1+(\p_r\widetilde{\Lambda})^2}dr=\int_0^s\sqrt{1+\f{4(1-\k^2)}{\pi^2\k^2}(\arctan r)^2}dr.$$
Hence
$$1\ge\tilde{\la}'(s)=\left(1+\f{4(1-\k^2)}{\pi^2\k^2}(\arctan \tilde{\la}(s))^2\right)^{-\f12}>\k,$$
and
$$\tilde{\la}''(s)=-\left(1+\f{4(1-\k^2)}{\pi^2\k^2}(\arctan \tilde{\la}(s))^2\right)^{-\f32}\f{4(1-\k^2)}{\pi^2\k^2}\arctan \tilde{\la}(s)\f{\tilde{\la}'(s)}{1+\tilde{\la}^2(s)}.$$
Clearly,
\begin{equation}\aligned\nonumber
\lim_{s\rightarrow\infty}\f{\tilde{\la}(s)}s=\lim_{s\rightarrow\infty}\tilde{\la}'(s)=\k,\qquad \mathrm{and}\qquad \lim_{s\rightarrow\infty}(s^2\tilde{\la}''(s))=-\f2{\pi}(1-\k^2).
\endaligned
\end{equation}

If $\{\p_\r\}$ and $\{e_\a\}_{\a=1}^{n-1}$ are an orthonormal basis of $\widetilde{\Si}$, then the sectional curvature of $\widetilde{\Si}$ is
\begin{equation}\aligned\nonumber
0<K(\p_\r,e_\a)=-\f{\tilde{\la}''}{\tilde{\la}}\thicksim\f{2(1-\k^2)}{\pi\k s^3}, \qquad K(e_\a,e_\be)=\f{1-\tilde{\la}'^2}{\tilde{\la}^2}\thicksim\f{1-\k^2}{\k^2 s^2}.
\endaligned
\end{equation}
Clearly,
$$\lim_{s\rightarrow0}\f{1-\tilde{\la}'^2(s)}{\tilde{\la}^2(s)}=\f{4(1-\k^2)}{\pi^2\k^2}>0.$$
Hence $\widetilde{\Si}=\{(x,\widetilde{\La}(x))|\ x\in\R^n\}$ has positive sectional curvature everywhere.

\begin{theorem}\label{PS}
Let $n\ge4$ and $\widetilde{\Si}=(\R^n,\tilde{\si})$ be a complete manifold with positive sectional curvature as  above. If
$$\f{2}{n-1}\sqrt{n-2}\le\k<1,$$
then any hyperplane through the origin in $\widetilde{\Si}=(\R^n,\tilde{\si})$ is area-minimizing.
\end{theorem}
\begin{proof}
Note $\k<\tilde{\la}'\le1$, then we can rewrite the metric $\tilde{\si}$
similar to \eqref{ss1}\eqref{ss2}\eqref{ss3}. Apply Theorem
\ref{Existence} to complete the proof.
\end{proof}
\begin{remark}
Our theorem above gives an example for the question (2) in \cite{An}, which is

If $N$ is a complete manifold of positive sectional curvature, does $N$ ever admit an area-minimizing hypersurface?
\end{remark}

Now scaling the manifold $MCS_\k$ yields $\ep^2MCS_\k$ for $\ep>0$, which is $\R^{n+1}$  endowed with the metric
\begin{equation}\aligned
\si_\ep=d\r^2+\ep^2\lambda^2\left(\f{\r}{\ep}\right)d\th^2
\endaligned
\end{equation}
in polar coordinates, where  $\la$ and $d\th^2$  as in \eqref{Sisi} and \eqref{lar}. Obviously $\ep\lambda\left(\f{\r}{\ep}\right)\ge\k\r$ and $\ep\lambda\left(\f{\r}{\ep}\right)$ converges to $\k\r$ uniformly as $\ep\rightarrow0$. Hence $\si_\ep$ converges to $\si_C$ as $\ep\rightarrow0$ (uniformly away from the origin), where $\si_C$ is the metric of $CS_\k$ defined in \eqref{siC}.

Now we can derive the result of F. Morgan in \cite{Mf}, obtained there
by a different method due to G. R. Lawlor \cite{L}.
\begin{proposition}\label{CSkmin}
Let $n\ge3$ and $\k\ge\f{2}{n}\sqrt{n-1}$. Then
any hyperplane in $(n+1)$-dimensional $CS_\k$ through the origin is area-minimizing.
\end{proposition}

\begin{proof}
Let $T_\ep$ denote the  hyperplane in $\ep^2MCS_\k$ corresponding to  $T\subset MCS_\k$ during the re-scaling procedure. Denote $T_0=\lim_{\ep\rightarrow0}T_\ep\subset \lim_{\ep\rightarrow0}\ep^2MCS_\k=CS_\k$ in the sense of Hausdorff distance. 

Now we consider a bounded domain $\Om_0\subset T_0$ and a subset $W_0\subset CS_\k$ with $\p\Om_0=\p W_0$. View $\Om_0$ as a set $\Om\subset\R^{n}$ with the induced metric from $T_0$, and $W_0$ as a set $W$ in $\R^{n+1}$ with the induced metric from $CS_\k$. Let $\Om_\ep$ be the set $\Om\subset\R^{n}$ with the induced metric from $T_\ep$, and $W_\ep$ be the set $W$ in $\R^{n+1}$ with the induced metric from $\ep^2MCS_\k$. Clearly, $\p\Om_\ep=\p W_\ep$, $\Om_0=\lim_{\ep\rightarrow0}\Om_\ep$ and $W_0=\lim_{\ep\rightarrow0}W_\ep$ in the sense of Hausdorff distance.

Let $\mathcal{H}^n(K)$ denote $n$-dimensional Hausdorff measure of $K$ for any set $K\subset CS_\k$, and $\mathcal{H}^n_\ep(K')$ denote $n$-dimensional Hausdorff measure of $K'$ for any set $K'\subset \ep^2MCS_\k$.
Note that $\ep\lambda\left(\f{\r}{\ep}\right)\rightarrow\k\r$ uniformly as $\ep\rightarrow0$, we have
\begin{equation}\aligned\label{HnOm0W0}
\mathcal{H}^n(\Om_0)=\lim_{\ep\rightarrow0}\mathcal{H}_\ep^n(\Om_\ep),\qquad\limsup_{\ep\rightarrow0}\mathcal{H}_\ep^n(W_\ep)\le \mathcal{H}^n(W_0).
\endaligned
\end{equation}
Since $T_\ep$ is area-minimizing in $\ep^2MCS_\k$, then
\begin{equation}\aligned\label{HnOmepWep}
\mathcal{H}_\ep^n(\Om_\ep)\le \mathcal{H}_\ep^n(W_\ep).
\endaligned
\end{equation}
Combining \eqref{HnOm0W0} and \eqref{HnOmepWep}, we obtain
$$\mathcal{H}^n(\Om_0)=\lim_{\ep\rightarrow0}\mathcal{H}_\ep^n(\Om_\ep)\le\limsup_{\ep\rightarrow0}\mathcal{H}_\ep^n(W_\ep)\le \mathcal{H}^n(W_0).$$
Hence $T_0$ is an area-minimizing hypersurface in $CS_\k$.
\end{proof}

Actually, here the number $\f{2}{n}\sqrt{n-1}$ is optimal. Namely, if
$\k<\f{2}{n}\sqrt{n-1}$ then every hyperplane in $CS_\k$ is no more
area-minimizing and even not stable. This also has been proved in
\cite{Mf}. Let us show this fact by using  the second variation
formula for the volume functional.

\begin{theorem}\label{PMC}
Let $\k\in(0,1]$ and $n\ge3$.
Any hyperplane in $(n+1)$-dimensional $CS_\k$ through the origin is area-minimizing if and only if
\begin{equation}\aligned\label{condk}
\k\ge\f{2}{n}\sqrt{n-1}.
\endaligned
\end{equation}
\end{theorem}
\begin{proof}
By Proposition \ref{CSkmin} we only need to prove that if \eqref{condk} fails to hold, any hyperplane in $CS_\k$ through the origin is not area-minimizing.
Let $X$ be a totally geodesic sphere in $S_\k$, then $X$ is  minimal in $S_\k$ and $P\triangleq CX$ is a hyperplane in $CS_\k$ through the origin. Clearly, $P$ is a minimal hypersurface in $CS_\k$ because it is totally geodesic. 
The second variation formula is (see \eqref{Indexphi})
\begin{equation}\aligned\label{Indexphi1}
I(\phi,\phi)=\int_\ep^1\bigg(\int_X\bigg(&-\De_X\phi-\f{n-1}{\k^2}\phi+(n-1)\phi\\
&-(n-1)\r\f{\p\phi}{\p\r}-\r^2\f{\p^2\phi}{\p \r^2}\bigg)\phi\ d\mu_X\bigg)\r^{n-3}d\r,
\endaligned
\end{equation}
where $\phi(x,t)\in C^2(X\times_{\r}\R)$.
Define a second order differential operator $L$ by
$$L=\r^2\f{\p^2}{\p \r^2}+(n-1)\r\f{\p}{\p\r}.$$
If $s=\log\r$, then
$$L=\f{\p^2}{\p s^2}+(n-2)\f{\p}{\p s}=e^{-\f{n-2}2s}\f{\p^2}{\p s^2}\left(e^{\f{n-2}2s}\cdot\right)-\f{(n-2)^2}4.$$
So the $k(k\ge1)$-th eigenvalue of $L$ on $[\ep,1]$ is
\begin{equation}\aligned\label{eigenL2}
\f{(n-2)^2}4+\left(\f{k\pi}{\log\ep}\right)^2
\endaligned
\end{equation}
with the $k$-th eigenfunction (see  \cite{Si} or \cite{X} for instance)
$$\r^{\f{2-n}2}\sin\left(\f{k\pi}{\log\ep}\log\r\right).$$

By the second variation formula \eqref{Indexphi1}, $P$ is stable if and only if
$$-\f{n-1}{\k^2}+n-1+\f{(n-2)^2}4\ge0,$$
i.e.,
$$\k\ge\f{2}{n}\sqrt{n-1}.$$

\end{proof}

\section{A class of  manifolds with non-negative Ricci curvature}

Let $N$ be an $(n+1)$-dimensional complete non-compact Riemannian manifold satisfying the following three conditions:

C1) Nonnegative Ricci curvature: Ric $\ge 0$;

C2) Euclidean volume growth:
\begin{equation*}
V_N\triangleq\lim_{r\rightarrow\infty}\f{Vol\big(B_r(x)\big)}{r^{n+1}}>0;
\end{equation*}

C3) Quadratic  decay of the curvature tensor: for sufficiently large $\r=d(x,p)$, the distance from a fixed point in $N$,
\begin{equation*}
|R(x)|\le\f c{\r^{2}(x)}.
\end{equation*}

By Gromov's compactness theorem \cite{GLP}, for any sequence $\bar{\ep}_i\rightarrow0$
there is a subsequence $\{\ep_i\}$ converging to zero such that  $\ep_i N=(N,\ep_i \bar{g},p)$ converges to a metric space $(N_{\infty},d_{\infty})$ with vertex $o$ in the pointed Gromov-Hausdorff sense. It is called the tangent cone at infinity.
$N_{\infty}\setminus\{o\}$ is a smooth manifold with $C^{1,\a}$
Riemannian metric $\bar{g}_{\infty}(0<\a<1)$ which is compatible with
the distance $d_{\infty}$. The precise statements were derived in
\cite{GW} and \cite{Ps} on the basis of the harmonic coordinate
constructions of \cite{JK}. In fact, $N_{\infty}\setminus\{o\}$ is a $D^{1,1}$-Riemannian manifold (see \cite{GW,Ps}). For any compact domain $K\subset N_{\infty}\setminus\{o\}$, there exists a diffeomorphism $\Phi_i:\ K\rightarrow \Phi_i(K)\subset\ep_iN$ such that $\Phi_i^*(\ep_i \bar{g})$ converges as $i\rightarrow\infty$ to $\bar{g}_{\infty}$ in the $C^{1,\a}$-topology on $K$.

Cheeger-Colding  (see Theorem 7.6 in \cite{ChC}) proved that under the conditions C1) and C2) the  cone $N_{\infty}$ is a metric cone, namely  $N_{\infty}=CX=\R^+\times_{\r}X$ for some $n$ dimensional smooth compact manifold $X$ with Diam $X\le\pi$  and the metric
$$\bar{g}_{\infty}=d\r^2+\r^2s_{ij}d\th_id\th_j$$
where $s_{ij}d\th_id\th_j$ is the metric of $X$ and $s_{ij}\in C^{1,\a}(X)$.
Let $\r_i$ be the distance function from $p$ to the considered point in $\ep_iN$.
Set $B^i_r(x)$ be the geodesic ball with radius $r$ and centered at
$x$ in $(N,\ep_i\bar{g})$, and $\mathcal{B}_r(x)$ be the geodesic ball
with radius $r$ and centered at $x$ in $N_{\infty}$. In particular, $X=\p \mathcal{B}_1(o)$.

Mok-Siu-Yau \cite{MSY} showed that if C1) and C2) hold, then
there exists the Green function $G(p,\cdot)$ on $N^{n+1}$ with $\lim_{r\rightarrow\infty}\sup_{\p B_r(p)}\left|Gr^{n-1}-1\right|=0$ and
\begin{equation}\label{Grn-1}
r^{1-n}\le G(p,x)\le C r^{1-n}
\end{equation}
for any $n\ge2$, $x\in\p B_r(p)$ and some constant $C$.
Set $\mathcal{R}=G^{\f1{1-n}}$, then
\begin{equation}\aligned\label{Deb}
\De_{N}\mR^2=2(n+1)|\bn \mR|^2.
\endaligned
\end{equation}

Under the additional condition C3),
Colding-Minicozzi  (see Corollary 4.11 in \cite{CM1}) showed that
\begin{equation}\aligned\label{Rr}
\limsup_{r\rightarrow\infty}\left(\sup_{\p B_r}\left|\f {\mathcal{R}}r-1\right|+\sup_{\p B_r}\Big|\left|\overline{\na} {\mathcal{R}}\right|-1\Big|\right)=0,
\endaligned
\end{equation}
and
\begin{equation}\aligned\label{HessR2}
\limsup_{r\rightarrow\infty}\left(\sup_{\p B_r}\left|\mathrm{Hess}_{\mR^2}-2\bar{g}\right|\right)=0,
\endaligned
\end{equation}
where $\mathrm{Hess}_{\mR^2}$ is the Hessian matrix of $\mR^2$ in $N$.
In particular, $|\overline{\na}\mathcal{R}|\le C(n,V_N)$ which is a constant depending only on $n,V_N$.

For any $f\in C^1(\p\mathcal{B}_1)$, we can extend $f$ to $N_{\infty}\setminus\{o\}$ by defining
$$f(\r_\infty\th)=f(\th)$$
for any $\r_\infty>0$ and $\th\in\p\mathcal{B}_1$. Let $\widetilde{\na}$ be the Levi-Civita connection of $N_\infty$, then
\begin{equation}\aligned\label{wnafpr}
\left\lan\widetilde{\na}f,\f{\p}{\p \r_\infty}\right\ran=0.
\endaligned
\end{equation}
For any $K_2>K_1>0$ and $\ep>0$, let $\Phi_i:
\overline{\mathcal{B}_{2K_2}}\setminus \mathcal{B}_{\f\ep2
  K_1}\rightarrow \Phi_i(\overline{\mathcal{B}_{2K_2}}\setminus
\mathcal{B}_{\f\ep2 K_1})\subset\ep_iN$ be a diffeomorphism such that
$\Phi_i^*(\ep_i \bar{g})$ converges as $i\rightarrow\infty$ to
$\bar{g}_{\infty}$ in the $C^{1,\a}$-topology on
$\overline{\mathcal{B}_{2K_2}}\setminus \mathcal{B}_{\f\ep2
  K_1}$. Moreover, $\Phi_i$ is $C^{2,\a}$-bounded relative to harmonic
coordinates with a bound independent of $i$ (see \cite{JK}).

Let $\bn^i$, $\De_N^i$, $\mathrm{Hess}^i$, $|\cdot|_i$, $\text{Ric}_{\ep_iN}$ and $R_{\ep_iN}$ be the Levi-Civita connection, Laplacian operator, Hessian matrix, the norm, Ricci curvature and curvature tensor of $\ep_iN$, respectively, then on $\ep_iN$ we have the relations
$$\aligned\r_i=&\ep_i^{\f12}\r,\qquad \bn^i=\bn,\qquad \De_N^i=\ep_i^{-1}\De_N,\qquad \mathrm{Hess}^i=\mathrm{Hess},\\
\text{Ric}_{\ep_iN}&=\ep_i^{-1}\text{Ric}, \qquad |R_{\ep_iN}|_i=\ep_i^{-1}|R|_i, \qquad d\mu_{\ep_iN}=\ep_i^{\f{n+1}{2}}d\mu_N
\endaligned$$
where $\r_i$ and $d\mu_{\ep_iN}$ are the distance function and volume element on $\ep_iN$, respectively, and $d\mu_N$ is the volume element on $N$. Let $\lan\cdot,\cdot\ran_i$ be the inner product corresponding to $|\cdot|_i$. We see that conditions C1), C2)  and C3) are all scaling
invariant.
Let
$$\widetilde{\mathcal{R}}_i=\sqrt{\ep_i}\mathcal{R}\qquad \mathrm{on}\ \ \ep_iN,$$
then
$$\De_{N}^i\widetilde{\mR}_i^2=\De_{N}\mR^2=2(n+1)|\bn\mR|^2=2(n+1)|\bn^i \widetilde{\mR}_i|_i^2$$
and so
$\widetilde{\mathcal{R}}_i^{1-n}$ is the Green function on $\ep_iN$. 
By \eqref{HessR2} we have
\begin{equation}\aligned\label{HessK21}
\limsup_{i\rightarrow\infty}\left(\sup_{B^i_{K_2}\setminus B^i_{\ep K_1}}\left|\mathrm{Hess}^i_{\widetilde{\mR}_i^2}-2\ep_i\bar{g}\right|_i\right)=0.
\endaligned
\end{equation}

The distance function $\r_i$ is a Lipschitz function with Lipschitz constant 1. 
If $\r_i$ is $C^1$-function at $x\in\ep_i N$, then $\bn^i\r_i(x)=\dot{\g^i_x}(\r_i(x))$, where $\g^i_x$ is a minimal normal geodesic
from $p$ to $x$. For any $y\in\g^i_x$ except $p,x$, the minimal normal geodesic joining $p$ and $y$ must a portion of $\g^i_x$. Because if there is a minimal normal geodesic $l^i_y$ joining $p$ and $y$, $l^i_y\cup\g^i_{x,y}$ is also a minimal normal geodesic, where $\g^i_{x,y}$ is the portion of $\g^i_x$ with boundary $\{x,y\}$. Then $l^i_y\cup\g^i_{x,y}$ is smooth and $\dot{l}^i_y=\dot{\g}^i_{x}(\r_i(y))$. So $l^i_y\subset\g^i_x$ by the uniqueness of geodesic equation, which is a differential equation of order 2. Therefore, we have $\bn^i\r_i(y)=\dot{\g^i_x}(\r_i(y))$.
When $\ep_i=1$, $\bn\r(z)$ corresponds to the normal geodesic
$\dot{\g_z}$ if $\r$ is $C^1$ at $z\in N$. 

Now if $x\in B^i_{K_2}\setminus B^i_{\ep K_1}$ and $\r_i$ is $C^1$-function at $x$, let $x=\g^i_x(t)$, $x_\ep=\g^i_x(t_\ep)\in\p B^i_{\ep K_1}\cap\g^i_x$, then for any parallel vector field $\xi$ along $\g^i_x$, we have
\begin{equation}\aligned
\overline{\na}^i_\xi\widetilde{\mathcal{R}}_i^2(x)-\overline{\na}^i_\xi\widetilde{\mathcal{R}}_i^2(x_\ep)
=\int_{t_\ep}^t\overline{\na}^i_{\dot{\g}^i_x}\overline{\na}^i_\xi\widetilde{\mathcal{R}}_i^2(\g^i_x(s))ds
=\int_{t_\ep}^t\mathrm{Hess}^i_{\widetilde{\mR}_i^2}\left(\bn^i\r_i,\xi\right)(\g^i_x(s))ds.
\endaligned
\end{equation}
Hence
\begin{equation}\aligned
\left|\overline{\na}^i_\xi\widetilde{\mathcal{R}}_i^2(x)-\bn^i_\xi\r_i^2(x)\right|_i
\le&\left|\overline{\na}^i_\xi\widetilde{\mathcal{R}}_i^2(x_\ep)-\bn^i_\xi\r_i^2(x_\ep)\right|_i\\
&+\int_{t_\ep}^t\left|\mathrm{Hess}^i_{\widetilde{\mR}_i^2}\left(\bn^i\r_i,\xi\right)(\g^i_x(s))-
\mathrm{Hess}^i_{\r_i^2}\left(\bn^i\r_i,\xi\right)(\g^i_x(s))\right|ds\\
\le&C\ep+\int_{t_\ep}^t\left|\mathrm{Hess}^i_{\widetilde{\mR}_i^2}\left(\bn^i\r_i,\xi\right)(\g^i_x(s))-
2\left\lan\bn^i\r_i(\g^i_x(s)),\xi\right\ran_i\right|ds\\
\le&C\ep+K_2\sup_{B^i_{K_2}\setminus B^i_{\ep K_1}}\left|\mathrm{Hess}^i_{\widetilde{\mR}_i^2}\left(\bn^i\r_i,\xi\right)-
2\left\lan\bn^i\r_i,\xi\right\ran_i\right|,
\endaligned
\end{equation}
where $C$ depends only on $K_1,K_2$ and the manifold $N$.
With \eqref{HessK21} we obtain
\begin{equation}\aligned
\limsup_{i\rightarrow\infty}\left(\sup_{Q_i\cap B^i_{K_2}\setminus B^i_{\ep K_1}}\left|\overline{\na}^i\widetilde{\mathcal{R}}_i^2(x)-\bn^i\r_i^2(x)\right|_i\right)\le  C\ep,
\endaligned
\end{equation}
where $Q_i$ is the set including all the regular points of $\r_i$ in $\ep_iN$ ($\r_i$ is $C^1$ at such points). Moreover, the $n$-dimensional Hausdorff measure $\mathcal{H}^n(\ep_iN\setminus Q_i)=0$.

Since the geodesics  $\g^i_x$ in $\ep_iN$ converge to a geodesic  in
$N_\infty$, with \eqref{wnafpr} and the continuity of $\overline{\na}^i(f\circ\Phi_i^{-1})$, we have
\begin{equation}\aligned\label{dfdRPhi}
\limsup_{i\rightarrow\infty}\sup_{B^i_{K_2}\setminus B^i_{\ep K_1}}
\left|\left\lan\overline{\na}^i(f\circ\Phi_i^{-1}),\overline{\na}^i\widetilde{\mathcal{R}}_i^2\right\ran_i\right|\le C_1\ep,
\endaligned
\end{equation}
and
\begin{equation}\aligned\label{dfRPhi}
\limsup_{i\rightarrow\infty}\sup_{B^i_{K_2}\setminus B^i_{\ep K_1}}
\left(\widetilde{\mathcal{R}}_i\left|\overline{\na}^i(f\circ\Phi_i^{-1})\right|_i\right)<\infty.
\endaligned
\end{equation}

Let $\Pi_i$ be the
identity map from $N$ to itself, the two copies of the manifold being regarded with respect to different
metrics:\ $\Pi_i:\ (N,\bar{g})\rightarrow\ep_iN=(N,\ep_i\bar{g})$. Now \eqref{dfdRPhi} and \eqref{dfRPhi} are equivalent to
\begin{equation}\aligned\label{dfdR2}
\limsup_{i\rightarrow\infty}\sup_{B_{\f{K_2}{\sqrt{\ep_i}}}\setminus B_{\f{\ep K_1}{\sqrt{\ep_i}}}}
\left|\left\lan\overline{\na}(f\circ\Phi_i^{-1}\circ\Pi_i),\overline{\na}\mathcal{R}^2\right\ran\right|\le C_1\ep,
\endaligned
\end{equation}
and
\begin{equation}\aligned\label{Rdf}
\limsup_{i\rightarrow\infty}\sup_{B_{\f{K_2}{\sqrt{\ep_i}}}\setminus B_{\f{\ep K_1}{\sqrt{\ep_i}}}}
\left(\mathcal{R}\left|\overline{\na}(f\circ\Phi_i^{-1}\circ\Pi_i)\right|\right)<\infty.
\endaligned
\end{equation}

The theory of integral currents in metric spaces was developed by Ambrosio and Kirchheim
in \cite{AK}. It provides a suitable notion of generalized surfaces in metric spaces, which extends the classical Federer-Fleming theory \cite{FF}.
We shall need the compactness Theorem (see Theorem 5.2 in \cite{AK}) and the closure Theorem (see Theorem 8.5 in \cite{AK}) for normal currents in a metric space $E$.
\begin{theorem}\label{ComT}
Let $(T_h)\subset \textbf{N}_k(E)$ be a bounded sequence of normal currents, and assume that for any integer $p\ge1$ there exists a compact set $K_p\subset E$ such that
$$||T_h||(E\setminus K_p)+||\p T_h||(E\setminus K_p)<\f1p\qquad \mathrm{for\ all}\ h\in \N.$$
Then, there exists a subsequence $(T_{h(n)})$ converging to a current $T\in\textbf{N}_k(E)$ satisfying
$$||T||\left(E\setminus \bigcup_{p=1}^\infty K_p\right)+||\p T||\left(E\setminus \bigcup_{p=1}^\infty K_p\right)=0.$$
\end{theorem}

\begin{theorem}\label{CloT}
Let $\mathcal{I}_k(E)$ be the class of integer-rectifiable currents in
$E$.
Let $(T_h)\subset \textbf{N}_k(E)$ be a sequence weakly converging to $T\in \textbf{N}_k(E)$.
Then, the conditions
$$T_h\in \mathcal{I}_k(E),\qquad \sup_{h\in\N}\textbf{N}(T_h)<\infty$$
imply $T\in \mathcal{I}_k(E)$.
\end{theorem}

Now let $M$ denote a minimal hypersurface in $N$ with the induced metric $g$ from $N$.
Since $N$ has nonnegative Ricci curvature, then $Vol(\p
B_r)\le\omega_nr^n$, where $\omega_n$ is the volume of the $n$-dimensional unit sphere in $\R^{n+1}$.
Suppose that $M$ has \emph{Euclidean volume growth at most}, namely,
\begin{equation}\label{EVGM}
\limsup_{r\rightarrow\infty}\left(r^{-n}\int_{M\cap B_r}1d\mu\right)<+\infty,
\end{equation}
where $d\mu$ is the volume element of $M$. Hence there is a smallest positive constant $V_M$ such that
$$\int_{M\cap B_r}1d\mu\le V_M\, r^n\qquad \ \mathrm{for\ any}\ r>0.$$

Denote $\ep_iM=(M,\ep_i g)$.
For any fixed $r>1$ let $\Phi_i: \overline{\mathcal{B}_{2r}}\setminus \mathcal{B}_{\f1{2r}}\rightarrow \Phi_i(\overline{\mathcal{B}_{2r}}\setminus \mathcal{B}_{\f1{2r}})\subset\ep_iN$ be a diffeomorphism such that $\Phi_i^*(\ep_i \bar{g})$ converges as $i\rightarrow\infty$ to $\bar{g}_{\infty}$ in the $C^{1,\a}$-topology on $\overline{\mathcal{B}_{2r}}\setminus \mathcal{B}_{\f1{2r}}$. We see that the minimality is also scaling invariant and $\ep_iM$ are also minimal hypersurfaces of $\ep_iN$. Since
 $$\int_{M\cap B_{2r}}1d\mu=\int_{0}^{2r}Vol\left(M\cap\p B_s\right)ds\le V_M\, 2^nr^n$$
which is scaling invariant, there exists a sequence $l_i\in(r,2r)$ such that $Vol\left(\ep_iM\cap\p B^i_{l_i}\right)+Vol\left(\ep_iM\cap\p B^i_{l_i^{-1}}\right)$ is uniformly bounded for every $i$.

Let $T_i=\ep_iM\cap \left(B^i_{l_i}\setminus B^i_{l_i^{-1}}\right)$, then $\Phi_i^{-1}(T_i)$ is a minimal hypersurface in $\big(N_\infty,\Phi_i^*(\ep_i\bar{g})\big)$ with the unit normal vector $\hat{\nu_i}$. If we change the metric $\Phi_i^*(\ep_i\bar{g})$ to $\bar{g}_{\infty}$, then we get a new metric $\tilde{g}_i$ on the hypersurface $\Phi_i^{-1}(T_i)$ induced from $\big(N_{\infty},\bar{g}_{\infty}\big)$. Set $\widetilde{T}_i=\big(\Phi_i^{-1}(T_i),\tilde{g}_i\big)$, and $\tilde{\nu_i}$ be the unit normal vector to the smooth hypersurface $\widetilde{T}_i$ in the metric space $\big(N_{\infty},\bar{g}_{\infty}\big)$.

$\Phi_i^*(\ep_i\bar{g})\rightarrow\bar{g}_{\infty}$ implies $\lim_{i\rightarrow\infty}\hat{\nu_i}=\lim_{i\rightarrow\infty}\tilde{\nu_i}\triangleq\nu_0$ and these two convergences are both uniform.
Then obviously
$$\mathcal{H}^n(\widetilde{T}_i)+\mathcal{H}^{n-1}(\p \widetilde{T}_i)$$
is uniformly bounded.
By Theorem \ref{ComT} and \ref{CloT} (see also \cite{S} for
compactness of currents in the Euclidean case), there is a subsequence of $\ep_{i_j}$ such that
\begin{equation}\aligned
\widetilde{T}_{i_j}\rightharpoonup T\qquad \mathrm{as}\ j\rightarrow\infty,
\endaligned
\end{equation}
where $T$ is an integer-rectifiable current in $N_{\infty}$. Denote $\widetilde{T}_{i_j}$ by $\widetilde{T}_i$ for simplicity.
Let $\mathcal{X}(\Om)$ be the set containing all smooth differential vector fields with compact support in $\Om$.
For any $\xi\in\mathcal{X}\left(\mathcal{B}_{2r}\setminus \mathcal{B}_{\f1r}\right)$ we have
\begin{equation}\aligned
\lim_{i\rightarrow\infty}\int_{\widetilde{T}_i}\lan\xi,\tilde{\nu_i}\ran d\tilde{\mu_i}=\int_{T}\lan\xi,\nu_\infty\ran d\mu_\infty,
\endaligned
\end{equation}
where $d\tilde{\mu_i}$ and $d\mu_\infty$ are the volume elements of $\widetilde{T}_i$ and $T$, respectively, and $\nu_\infty$ is the unit normal vector of $T$. Since $\hat{\nu_i}\rightarrow\nu_0$ and $\tilde{\nu_i}\rightarrow\nu_0$ uniformly, then we have
\begin{equation}\aligned
\int_{T}\lan\xi,\nu_\infty\ran d\mu_\infty=\lim_{i\rightarrow\infty}\int_{\Phi_i^{-1}(T_i)}\lan\xi,\hat{\nu_i}\ran d\hat{\mu_i}=\lim_{i\rightarrow\infty}\int_{T_i}\lan\xi\circ\Phi_i^{-1},\nu_i\ran_i d\mu_i,
\endaligned
\end{equation}
where $d\hat{\mu_i}$ and $d\mu$ are the volume elements of $\Phi_i^{-1}(T_i)$ and $T_i$, respectively. Then we conclude that
\begin{equation}\aligned\label{PhiMT}
T_i=\ep_{i}M\bigcap B^{i}_{l_{i}}\setminus B^{i}_{l_i^{-1}}\rightharpoonup T\qquad \mathrm{as}\ i\rightarrow\infty.
\endaligned
\end{equation}

\section{Non-existence of area-minimizing hypersurfaces}

Before we can  prove our main results, we still need volume estimates
for minimal hypersurfaces. In fact, these results are interesting in their
own right.
\begin{theorem}\label{Volest}
let $M$ be a complete $n$-dimensional minimal hypersurface  in a complete non-compact
Riemannian manifold $N$ satisfying  conditions C1), C2), C3). Then
\begin{enumerate}
\item every end $E$ of $M$ has infinite volume;
\item if M is a proper immersion, then $M$   has Euclidean volume
growth at least,
\begin{equation}\aligned\label{MEleast}
\liminf_{r\rightarrow\infty}\left(\f1{r^n}\int_{M\cap B_r(p)}1d\mu\right)>0,\qquad for\ any\ p\in N;
\endaligned
\end{equation}
\item If $M$ has  at most Euclidean volume growth, i.e.,
\begin{equation}\label{MEmost}
\limsup_{r\rightarrow\infty}\left(r^{-n}\int_{M\cap B_r}1d\mu\right)<\infty,
\end{equation}
then $M$ is a proper immersion.
\end{enumerate}
\end{theorem}

\begin{proof}
For any $0<\de\le1$, set $\Om=\left(\f{\sqrt{c}}{\de}+1\right)$ with $c$ as in axiom C3). For any fixed point $p\in N$ and arbitrary $q\in\p B_{\Om r}(p)$, we have
$$d(p,x)\ge\f{\sqrt{c}}{\de}\ r,\qquad \mathrm{for\ any}\ \ x\in B_r(q).$$
Then by condition C3) the sectional curvature satisfies
\begin{equation}\aligned
|K_N(x)|\le\f{\de^2}{r^2}, \qquad \mathrm{for\ any}\ \ x\in B_r(q).
\endaligned
\end{equation}
Note $Vol(B_s(q))\ge V_N s^{n+1}$ for any $s>0$ as implied by conditions C1),C2) by virtue of the Bishop-Gromov theorem.

Now, we claim that by \cite{CGT}, for sufficiently small $\de$ (depending only on $n,c,V_{N}$) the injectivity radius at $q$ satisfies $i(q)\ge r$ for every $r>0$ and $q\in\p B_{\Om r}(p)$. Arguing by contradiction, let us assume instead the existence of sequences $\de_j\rightarrow0$, $r_j>0$ and a sequence of $q_j\in\p B_{\left(1+\sqrt{c}\de_j^{-1}\right)r_j}(p)$ such that the injectivity radius at $q_j$ on $N$ satisfies $i(q_j)<r_j$.
Let $\bar{r}_j=\left(1+\f{\sqrt{c}}{2\de_{j}}\right)r_{j}$, and $\mathcal{Z}_j=\f1{\bar{r}_j}B_{\bar{r}_j}(q_j)=\left(B_{\bar{r}_{j}}(q_{j}),\f1{\bar{r}_{j}}\bar{g},q_{j}\right)$. Now $\mathcal{Z}_j$ has nonnegative Ricci curvature, and $Vol(\widehat{B}^{\mathcal{Z}_j}_s(q_{j}))\ge V_N s^{n+1}$ for any $0<s<1$, where $\widehat{B}^{\mathcal{Z}_j}_s(q_{j})$ is the geodesic ball centered at $q_j$ with radius $s$ in $\mathcal{Z}_j$. The sectional curvature tensor of $\mathcal{Z}_j$ satisfies
$$|K_{\mathcal{Z}_j}(x)|\le\de_j^2, \qquad \mathrm{for\ any}\  x\in \mathcal{Z}_j.$$
Moreover, the injectivity radius at $q_{j}$ on $\mathcal{Z}_j$ satisfies $i(q_{j})<\f{r_j}{\bar{r}_j}=\left(1+\f{\sqrt{c}}{2\de_{j}}\right)^{-1}$. By Gromov's compactness theorem \cite{GLP}, there is a subsequence $j_k$ such that $\mathcal{Z}_{j_k}$ converges to a unit ball with the standard Euclidean metric in the pointed Gromov-Hausdorff sense.  So we assume that  $\mathcal{Z}_{j_k}$ is connected without loss of generality. Suppose $y_k\in\mathcal{Z}_{j_k}$ satisfying $d_{\mathcal{Z}_{j_k}}(y_k,q_{j_k})=i(q_{j_k})$, where $d_{\mathcal{Z}_{j_k}}$ is the distance function on $\mathcal{Z}_{j_k}$. Obviously, $y_k$ and $q_{j_k}$ are not conjugate by $|K_{\mathcal{Z}_j}(x)|\le\de_j^2\rightarrow0$ as $j\rightarrow\infty$. By the proof of Klingenberg's lemma on the injectivity radius (see page 158 in \cite{GHL} for instance), there are two normal minimal geodesics $\Gamma^1_k,\G^2_k\subset\mathcal{Z}_{j_k}$ both joining $q_{j_k},y_k$ such that $\dot{\Gamma}^1_k(y_k)=-\dot{\Gamma}^2_k(y_k)$, i.e., $\Gamma^1_k\cup\G^2_k$ is a geodesic loop on $q_{j_k}$. By Theorem 4.3 in \cite{CGT}, there is a uniform positive lower bound independent of $k$ for the length of the loop. Namely, $d_{\mathcal{Z}_{j_k}}(y_k,q_{j_k})=i(q_{j_k})$ has a uniform positive lower bound, which violates $i(q_{j_k})<\left(1+\f{\sqrt{c}}{2\de_{j_k}}\right)^{-1}\rightarrow0$ as $k\rightarrow\infty$. Hence we complete the proof of the above claim.


So $\r_q(x)$ is smooth for $x\in B_r(q)\setminus\{q\}$.
Assume $q\in M$. Let $\{e_i\}$ be a local orthonormal frame field  of $M$. Then
\begin{equation}\aligned
\De_M\r_q^2=&\sum_{i=1}^n\left(\na_{e_i}\na_{e_i}\r_q^2-\left({\na_{e_i}e_i}\right)\r_q^2\right)\\
=&\sum_{i=1}^n\left(\bn_{e_i}\bn_{e_i}\r_q^2-\left({\bn_{e_i}e_i}\right)\r_q^2\right)
+\sum_{i=1}^n\left({\bn_{e_i}e_i}-\na_{e_i}e_i\right)\bar{\r}_q^2\\
=&\sum_{i=1}^n\mathrm{Hess}_{\r_q^2}(e_i,e_i).
\endaligned
\end{equation}
For any $\xi\in\G(TN)$ we denote $\xi^T_q=\xi-\left\lan \xi,\f{\p}{\p\r_q}\right\ran \f{\p}{\p\r_q}$.
Combining $\mathrm{Hess}_{\r_q^2}\left(\xi^T_q,\f{\p}{\p\r_q}\right)=0$ and $\mathrm{Hess}_{\r_q^2}\left(\f{\p}{\p\r_q},\f{\p}{\p\r_q}\right)=2$, we obtain
\begin{equation}\aligned
\mathrm{Hess}_{\r_q^2}(e_i,e_i)=&\mathrm{Hess}_{\r_q^2}\left((e_i)^T_q,(e_i)^T_q\right)+2\left\lan e_i,\f{\p}{\p\r_q}\right\ran^2\\
=&2\r_q\mathrm{Hess}_{\r_q}\left((e_i)^T_q,(e_i)^T_q\right)+2\left\lan e_i,\f{\p}{\p\r_q}\right\ran^2.
\endaligned
\end{equation}
The injectivity radius $i(q)\ge r$ implies that $\r_q$ is a smooth function on $B_r(q)\setminus\{q\}$.
By the Hessian comparison theorem and the sectional curvature $K_N(x)\le\f{\de^2}{r^2}$ for any $x\in B_r(q)$, for any $\xi\bot\f{\p}{\p\r_q}$ we have
$$\mathrm{Hess}_{\r_q}(\xi,\xi)\ge \f{\de}r\cot{\left(\f{\de\r_q}r\right)}|\xi|^2.$$
Since $\f{\de\r_q}r\cot{\left(\f{\de\r_q}r\right)}\le1$ for $\r_q\le r$ with sufficiently small $\de$, then
\begin{equation}\aligned\label{DeMrq2}
\De_M\r_q^2\ge&2\r_q\sum_{i=1}^n\f{\de}r\cot{\left(\f{\de\r_q}r\right)}\left|(e_i)^T_q\right|^2+2\sum_{i=1}^n\left\lan e_i,\f{\p}{\p\r_q}\right\ran^2\\
\ge&2\f{\de\r_q}r\cot{\left(\f{\de\r_q}r\right)}\sum_{i=1}^n\left|(e_i)^T_q\right|^2+2\f{\de\r_q}r\cot{\left(\f{\de\r_q}r\right)}\sum_{i=1}^n\left\lan e_i,\f{\p}{\p\r_q}\right\ran^2\\
=&2n\f{\de\r_q}r\cot{\left(\f{\de\r_q}r\right)}.
\endaligned
\end{equation}
For any $t\in[0,1)$ we have $\cos t\ge1-t$, then
$$\left(\tan t-\f t{1-t}\right)'=\f1{\cos^2 t}-\f1{(1-t)^2}\le0.$$
So on $[0,1)$
$$\tan t\le\f t{1-t}.$$
Denote the extrinsic ball $D_s(q)=B_s(q)\cap M$. Hence on $D_r(q)$ we have
\begin{equation}\label{Lro}
\De_M\r_{q}^2(x)\ge2n\left(1-\f\de r\r_q(x)\right)=2n-\f{2n\de\r_q(x)}{r}.
\end{equation}
Let $\r^M_q$ and $B_s^M(q)$  be the distance function from $q$ and the geodesic ball with radius $s$ and centered at $q$ in $M$. Obviously, the intrinsic ball $B_s^M(q)\subset D_s(q)$ for any $s\in(0,r)$ and (\ref{Lro}) is valid on  $B_r^M(q)$.

Integrating (\ref{Lro}) by parts on $B^M_s(q)$ yields
\begin{equation}
2n\int_{B_s^M(q)}\left(1-\f{\de\r_q}{r}\right)
\le\int_{B_s^M(q)}\De_M\r_{q}^2=\int_{\p B_s^M(q)}\na\r_{q}^2\cdot\nu
\le 2s\int_{\p B_s^M(q)}|\na\r_{q}|,
\end{equation}
where $\nu$ is the normal vector to $\p B_s^M(q)$.
Then
\begin{equation}\aligned
\f{\p}{\p s}\left(s^{-n}\int_{ B_s^M(q)}1\right)
=&-n s^{-n-1}\int_{B_s^M(q)}1+s^{-n}\int_{\p  B_s^M(q)}1\\
\ge&-n s^{-n-1}\int_{B_s^M(q)}1+s^{-n}\int_{\p  B_s^M(q)}|\na\r_q|\\
\ge& -n s^{-n-1}\int_{B_s^M(q)}1+n s^{-n-1}\int_{B_s^M(q)}\left(1-\f{\de\r_q}{r}\right)\\
=&-\f{n\de}{r}s^{-n}\int_{B_s^M(q)}1.
\endaligned
\end{equation}
Integrating the above inequality implies for $0<s\le r$
\begin{equation}\aligned\label{volBsMq}
\text{vol}(B_s^M(q))\triangleq\int_{B_s^M(q)}1\ge\f{\omega_{n-1}}ns^{n}e^{-\f{n\de s}r}\ge \f{\omega_{n-1}}ns^{n}e^{-n\de}.
\endaligned
\end{equation}
Here $\omega_{n-1}$ is the measure of the standard $(n-1)$-dimensional unit sphere in Euclidean space.

(i) Let $E$ be an and of $M$. If $E$ is not contained in any bounded
domain in $N$, then we choose $r$ large enough and some $q\in\p B_{\Om r}(p)\cap M$. By \eqref{volBsMq}, $E$ then has infinite volume.

Now we suppose that $E\subset B_{R_0}(p)$ for some constant $R_0>0$.
Recalling \eqref{volBsMq}, there is a constant $r_0>0$ so that for any $0<r\le r_0$ and $z\in E$ we have a constant $C_0>0$ such that
\begin{equation}\aligned
\text{vol}(B_{r}^M(z))\ge C_0r^{n}.
\endaligned
\end{equation}
Since $E$ is noncompact, then we can choose a sequence $\{z_i\}$ such that $B_{r_0}^M(z_i)\cap B_{r_0}^M(z_j)\neq\emptyset$ for $i\neq j$. Hence
$$\text{vol}(E)\ge\sum_i\text{vol}(B_{r_0}^M(z_i))\ge C_0\sum_ir_0^n=\infty.$$


(ii) Since $B_s^M(q)\subset D_s(q)$ for any point $q\in\p B_{\Om r}(p)$ and any $s\in(0,r)$, then with \eqref{volBsMq} we obtain
\begin{equation}\aligned
\int_{D_s(q)}1\ge \f{\omega_{n-1}}ns^{n}e^{-n\de}\qquad \mathrm{for\ every}\ s\in(0,r].
\endaligned
\end{equation}
Hence we conclude that \eqref{MEleast} holds.

(iii) If $M$ is not a proper immersion into $N$, there exist an end $E\subset M$ and a constant $r_0$, such that $E\subset B_{r_0}(p)$.
The assumption that $M$ has at most Euclidean volume growth  implies $M$ has finite volume, which contradicts  the results in (i).
\end{proof}

Let $M$ be a minimal hypersurface in $N$ with Euclidean volume growth at most. Combining \eqref{Grn-1}\eqref{Rr} and the definition of $\mathcal{R}$,
the quantity
$$r^{-n}\int_{M\cap\{\mathcal{R}\le r\}}|\overline{\na}\mathcal{R}|^2d\mu$$
is uniformly bounded for any $r\in(0,\infty)$, then there exists a sequence $r_i\rightarrow\infty$ such that
\begin{equation}\label{LNR}\limsup_{r\rightarrow\infty}\left(r^{-n}\int_{M\cap\{\mathcal{R}\le r\}}|\overline{\na}\mathcal{R}|^2d\mu\right)
=\lim_{r_i\rightarrow\infty}\left(r_i^{-n}\int_{{M\cap\{\mathcal{R}\le r_i\}}}|\overline{\na}\mathcal{R}|^2d\mu\right).\end{equation}

\begin{lemma}\label{limfdR}
There is a sequence $\de_i\rightarrow0^+$ such that for any constants $K_2>K_1>0$ and $\ep\in(0,1)$ and any bounded Lipschitz function $f$ on $N\setminus B_1$ we have
\begin{equation}\aligned
&\limsup_{i\rightarrow\infty}\left|\left(\f{\de_i}{K_2r_i}\right)^{n}\int_{M\cap\{\mathcal{R}\le \f{K_2r_i}{\de_i}\}}f|\overline{\na}\mathcal{R}|^2-\left(\f{\de_i}{K_1r_i}\right)^{n}\int_{M\cap\{\mathcal{R}\le \f{K_1 r_i}{\de_i}\}}f|\overline{\na}\mathcal{R}|^2\right|\\
\le&C\ep^n\sup_{N\setminus B_1}|f|+\limsup_{i\rightarrow\infty}\int_{\f{K_1 r_i}{\de_i}}^{\f{K_2r_i}{\de_i}}\left(s^{-n-1}\int_{M\cap\{\f{\ep K_1r_i}{\de_i}<\mathcal{R}\le s\}}\mathcal{R}\na f\cdot\na\mathcal{R}\right)ds.
\endaligned
\end{equation}
\end{lemma}
\begin{proof}
Let $\{e_i\}$ be an orthonormal basis of $TM$ and $\nu$ be the unit normal vector of $M$. Then by \eqref{Deb} we have
\begin{equation}\aligned
\De_M\mR^2=&\sum_{i=1}^n\left(\na_{e_i}\na_{e_i}\mR^2-\left({\na_{e_i}e_i}\right)\mR^2\right)\\
=&\sum_{i=1}^n\left(\bn_{e_i}\bn_{e_i}\mR^2-\left({\bn_{e_i}e_i}\right)\mR^2\right)
+\sum_{i=1}^n\left({\bn_{e_i}e_i}-\na_{e_i}e_i\right)\mR^2\\
=&\De_{N}\mR^2-\mathrm{Hess}_{\mR^2}(\nu,\nu)\\
=&2(n+1)|\bn\mR|^2-\mathrm{Hess}_{\mR^2}(\nu,\nu).
\endaligned
\end{equation}
By \eqref{HessR2} and \eqref{Rr} there exists a sequence $\de_i\rightarrow0^+$ such that on $M\setminus B_{\sqrt{r_i}}$ we have
\begin{equation}\aligned\label{deidr5.16}
\left|\De_M\mathcal{R}^2-2n|\overline{\na}\mathcal{R}|^2\right|\le2\de_i|\overline{\na}\mathcal{R}|^2.
\endaligned
\end{equation}
For any $s\ge\a_ir_i^{\f12}$ with $\a_i\ge1$ and $f\in\mathrm{ Lip}(N\setminus B_1)$, integrating by parts yields
\begin{equation}\aligned
2s\int_{M\cap\{\mathcal{R}=s\}}f|\na\mathcal{R}|&-2\a_ir_i^{\f12}\int_{M\cap\{\mathcal{R}=\a_ir_i^{\f12}\}}f|\na\mathcal{R}|
=\int_{M\cap\{\a_ir_i^{\f12}<\mathcal{R}\le s\}}\mathrm{div}_M\left(f\na\mathcal{R}^2\right)\\
=&\int_{M\cap\{\a_ir_i^{\f12}<\mathcal{R}\le s\}}\na f\cdot\na\mathcal{R}^2+\int_{M\cap\{\a_ir_i^{\f12}<\mathcal{R}\le s\}}f\De_M\mathcal{R}^2.
\endaligned
\end{equation}
Hence,
\begin{equation}\aligned\label{fdRa}
&\f{\p}{\p s}\left(s^{-n}\int_{M\cap\{\mathcal{R}\le s\}}f|\overline{\na}\mathcal{R}|^2\right)\\
=&-n s^{-n-1}\int_{M\cap\{\mathcal{R}\le s\}}f|\overline{\na}\mathcal{R}|^2+s^{-n}\int_{M\cap\{\mathcal{R}=s\}}f\f{|\overline{\na}\mathcal{R}|^2}{|\na\mathcal{R}|}\\
=&-n s^{-n-1}\int_{M\cap\{\mathcal{R}\le s\}}f|\overline{\na}\mathcal{R}|^2+s^{-n}\int_{M\cap\{\mathcal{R}=s\}}f|\na\mathcal{R}|
+s^{-n}\int_{M\cap\{\mathcal{R}=s\}}f\f{\lan\overline{\na}\mathcal{R},\nu\ran^2}{|\na\mathcal{R}|}\\
=&-n s^{-n-1}\int_{M\cap\{\mathcal{R}\le s\}}f|\overline{\na}\mathcal{R}|^2+\f12s^{-n-1}\int_{M\cap\{\a_ir_i^{\f12}<\mathcal{R}\le s\}}f\De_M\mathcal{R}^2\\
&+\a_ir_i^{\f12} s^{-n-1}\int_{M\cap\{\mathcal{R}=\a_i r_i^\f{1}{2}\}}f|\na\mathcal{R}|+s^{-n-1}\int_{M\cap\{\a_ir_i^{\f12}<\mathcal{R}\le s\}}\mathcal{R}\na f\cdot\na\mathcal{R}\\
&+s^{-n}\int_{M\cap\{\mathcal{R}=s\}}f\f{\lan\overline{\na}\mathcal{R},\nu\ran^2}{|\na\mathcal{R}|}\\
=&-n s^{-n-1}\int_{M\cap\{\mathcal{R}\le \a_ir_i^{\f12}\}}f|\overline{\na}\mathcal{R}|^2+\f12s^{-n-1}\int_{M\cap\{\a_ir_i^{\f12}<\mathcal{R}\le s\}}f\left(\De_M\mathcal{R}^2-2n|\overline{\na}\mathcal{R}|^2\right)\\
&+\a_ir_i^{\f12} s^{-n-1}\int_{M\cap\{\mathcal{R}=\a_i r_i^\f{1}{2}\}}f|\na\mathcal{R}|+s^{-n-1}\int_{M\cap\{\a_ir_i^{\f12}<\mathcal{R}\le s\}}\mathcal{R}\na f\cdot\na\mathcal{R}\\
&+s^{-n}\int_{M\cap\{\mathcal{R}=s\}}f\f{\lan\overline{\na}\mathcal{R},\nu\ran^2}{|\na\mathcal{R}|}.
\endaligned
\end{equation}
Denote
$$V_M\triangleq\sup_{r>0}\left(r^{-n}\int_{M\cap\{\mathcal{R}\le r\}}|\overline{\na}\mathcal{R}|^2d\mu\right).$$
Select $f\equiv1$ and $\a_i=1$ in (\ref{fdRa}) and integrate.
Then for any $r\ge\sqrt{r_i}$ there is a constant $C$ depending only on $N$ and $V_M$ such that
\begin{equation}\aligned
&\left(\de_i^{-2}r\right)^{-n}\int_{M\cap\{\mathcal{R}\le \de_i^{-2}r\}}|\overline{\na}\mathcal{R}|^2-r^{-n}\int_{M\cap\{\mathcal{R}\le r\}}|\overline{\na}\mathcal{R}|^2\\
\ge&-nCr_i^{\f n2}\int_r^{\de_i^{-2}r}s^{-n-1}ds-C\de_i\int_r^{\de_i^{-2}r}\f1sds
+\int_r^{\de_i^{-2}r}\left(s^{-n}\int_{M\cap\{\mathcal{R}=s\}}\f{\lan\overline{\na}\mathcal{R},\nu\ran^2}{|\na\mathcal{R}|}\right)ds\\
\ge&-C\f{r_i^{\f n2}}{r^n}+2C\de_i\log \de_i+\int_r^{\de_i^{-2}r}\left(s^{-n}\int_{M\cap\{\mathcal{R}=s\}}\f{\lan\overline{\na}\mathcal{R},\nu\ran^2}{|\na\mathcal{R}|}\right)ds.
\endaligned
\end{equation}
Choose $r=r_i$ in the above inequality and let $i$ go to infinity, then we obtain
\begin{equation}\aligned
&\limsup_{r\rightarrow\infty}\left(r^{-n}\int_{M\cap\{\mathcal{R}\le r\}}|\overline{\na}\mathcal{R}|^2\right)-\lim_{i\rightarrow\infty}\left(r_i^{-n}\int_{M\cap\{\mathcal{R}\le r_i\}}|\overline{\na}\mathcal{R}|^2\right)\\
\ge&\lim_{i\rightarrow\infty}\left(\left(\de_i^{-2}r_i\right)^{-n}\int_{M\cap\{\mathcal{R}\le \de_i^{-2}r_i\}}|\overline{\na}\mathcal{R}|^2\right)-\lim_{i\rightarrow\infty}\left(r_i^{-n}\int_{M\cap\{\mathcal{R}\le r_i\}}|\overline{\na}\mathcal{R}|^2\right)\\
\ge&\lim_{i\rightarrow\infty}\int_{r_i}^{\de_i^{-2}r_i}
\left(s^{-n}\int_{M\cap\{\mathcal{R}=s\}}\f{\lan\overline{\na}\mathcal{R},\nu\ran^2}{|\na\mathcal{R}|}\right)ds.
\endaligned
\end{equation}
Combining the above inequality and (\ref{LNR}) imply
\begin{equation}\aligned\label{deirideRv111}
\lim_{i\rightarrow\infty}\int_{r_i}^{\de_i^{-2}r_i}\left(s^{-n}\int_{M\cap\{\mathcal{R}=s\}}\f{\lan\overline{\na}\mathcal{R},\nu\ran^2}{|\na\mathcal{R}|}\right)ds=0,
\endaligned
\end{equation}
namely, by the coarea formula
\begin{equation}\aligned\label{deiridRnu}
\lim_{i\rightarrow\infty}\int_{M\cap\{r_i<\mathcal{R}\le\de_i^{-2}r_i\}}\f{\lan\overline{\na}\mathcal{R},\nu\ran^2}{\mathcal{R}^n}=0.
\endaligned
\end{equation}

Set $|f|_0\triangleq\sup_Nf<\infty$ and $\a_i=\ep K_1r_i^{\f12}\de_i^{-1}$ for any small $\ep\in(0,1)$ in \eqref{fdRa}, then from \eqref{fdRa}, for any $r\ge\ep r_i\de_i^{-1}$ we obtain
\begin{equation}\aligned
&\left|(K_2r)^{-n}\int_{M\cap\{\mathcal{R}\le K_2r\}}f|\overline{\na}\mathcal{R}|^2-(K_1r)^{-n}\int_{M\cap\{\mathcal{R}\le K_1r\}}f|\overline{\na}\mathcal{R}|^2\right|\\
\le&nC|f|_0\left(\f{\ep K_1r_i}{\de_i}\right)^n\int_{K_1r}^{K_2r}s^{-n-1}ds+C\de_i|f|_0\int_{K_1r}^{K_2r}\f1sds\\
&+|f|_0\left(\f{\ep K_1r_i}{\de_i}\int_{M\cap\{\mathcal{R}=\f{\ep K_1r_i}{\de_i}\}}|\na\mathcal{R}|\right)\int_{K_1r}^{K_2r}s^{-n-1}ds\\
&+\int_{K_1r}^{K_2r}\left(s^{-n-1}\int_{M\cap\{\f{\ep K_1r_i}{\de_i}<\mathcal{R}\le s\}}\mathcal{R}\na f\cdot\na\mathcal{R}\right)ds\\
&+|f|_0\int_{K_1r}^{K_2r}\left(s^{-n}\int_{M\cap\{\mathcal{R}=s\}}\f{\lan\overline{\na}\mathcal{R},\nu\ran^2}{|\na\mathcal{R}|}\right)ds\\
\le&C|f|_0\f {\ep^nr_i^n}{\de_i^nr^n}+C\de_i|f|_0\log\f{K_2}{K_1}+\int_{K_1r}^{K_2r}\left(s^{-n-1}\int_{M\cap\{\f{\ep K_1r_i}{\de_i}<\mathcal{R}\le s\}}\mathcal{R}\na f\cdot\na\mathcal{R}\right)ds\\
&+\f{|f|_0}{nK_1^nr^n}\f{\ep K_1r_i}{\de_i}\int_{M\cap\{\mathcal{R}=\f{\ep K_1r_i}{\de_i}\}}|\na\mathcal{R}|+|f|_0\int_{K_1r}^{K_2 r}\left(s^{-n}\int_{M\cap\{\mathcal{R}=s\}}\f{\lan\overline{\na}\mathcal{R},\nu\ran^2}{|\na\mathcal{R}|}\right)ds,
\endaligned
\end{equation}
where we have used \eqref{deidr5.16} in the second inequality, and the definition of $\mathcal{R}$ in section 4. Since
$$ \f{\ep K_1r_i}{\de_i}\int_{M\cap\{\mathcal{R}=\f{\ep K_1r_i}{\de_i}\}}|\na\mathcal{R}|=\f12\int_{M\cap\{\mathcal{R}\le \f{\ep K_1r_i}{\de_i}\}}\De_M\mathcal{R}^2,$$
we get
\begin{equation}\aligned
&\left|(K_2r)^{-n}\int_{M\cap\{\mathcal{R}\le K_2r\}}f|\overline{\na}\mathcal{R}|^2-(K_1r)^{-n}\int_{M\cap\{\mathcal{R}\le K_1r\}}f|\overline{\na}\mathcal{R}|^2\right|\\
\le&C|f|_0\f {\ep^nr_i^n}{\de_i^nr^n}+C\de_i|f|_0\log\f{K_2}{K_1}+\int_{K_1r}^{K_2r}\left(s^{-n-1}\int_{M\cap\{\f{\ep K_1r_i}{\de_i}<\mathcal{R}\le s\}}\mathcal{R}\na f\cdot\na\mathcal{R}\right)ds\\
&+C_1\f{|f|_0}{2n\de_i^nr^n}\ep^nr_i^n
+|f|_0\int_{K_1r}^{K_2r}\left(s^{-n}\int_{M\cap\{\mathcal{R}=s\}}\f{\lan\overline{\na}\mathcal{R},\nu\ran^2}{|\na\mathcal{R}|}\right)ds
\endaligned
\end{equation}
for some $C_1>0$. Letting $r=\f{r_i}{\de_i}$ and $i\rightarrow\infty$ in the above inequality, then the conclusion follows by \eqref{deirideRv111}.
\end{proof}

Let $\ep_i=\de_i^2r_i^{-2}$ and suppose that $\ep_iN$ converges to $(N_\infty,d_\infty)$ without loss of generality.
Let $\ep_iM=(M,\ep_i g)$ and $D^i_r(x)=\ep_iM\cap B^i_r(x)$. Clearly, $\ep_iM$ is still a minimal hypersurface in $\ep_iN$ with $Vol\left(\ep_iM\cap B^i_r(p)\right)\le V_M\,r^n$.
\begin{lemma}\label{coneCY}
There exists a subsequence $\{\ep_{i_j}\}\subset\{\ep_i\}$ such that
$\ep_{i_j}M$ converges to a cone $CY=\R^+\times_{\r}Y$ in $N_{\infty}$, where $Y\subset \p\mathcal{B}_1(o)$ is an $(n-1)$-dimensional Hausdorff set with $\mathcal{H}^{n-1}(Y)>0$.
\end{lemma}
\begin{proof}

Note \eqref{PhiMT}. By choosing a diagonal sequence, we can assume
$$\Phi_{i_j}^{-1}\left(\ep_{i_j}M\bigcap \overline{B^{i_j}}_r\setminus B^{i_j}_{\f1r}\right)\rightharpoonup T\qquad \mathrm{as}\ j\rightarrow\infty,$$
for any $r>1$, where $T$ is an integer-rectifiable current in $N_{\infty}$. For convenience, we still write $\ep_i$  instead of $\ep_{i_j}$.

Let $f$ be a homogenous function in $C^1(N_{\infty}\setminus\{o\})$,
that is,
$$f(\r\th)=f(\th)$$
for any $\r>0$ and $\th\in\p\mathcal{B}_1$. Let $\Pi_i$ be the map from $(N,\bar{g})$ to $\ep_iN=(N,\ep_i\bar{g},p)$ defined before, then both of \eqref{dfdR2} and \eqref{Rdf} hold. Now we can extend the function $f\circ\Phi_i^{-1}\circ\Pi_i$ to a uniformly bounded function $F_i$ in $B_{\f{K_2r_i}{\de_i}}=B_{\f{K_2}{\sqrt{\ep_i}}}$ with $F_i=f\circ\Phi_i^{-1}\circ\Pi_i$ on $B_{\f{K_2r_i}{\de_i}}\setminus B_{\f{\ep K_1r_i}{\de_i}}=B_{\f{K_2}{\sqrt{\ep_i}}}\setminus B_{\f{\ep K_1}{\sqrt{\ep_i}}}$. Note \eqref{Grn-1} and the definition of $\mathcal{R}$.
Hence for sufficiently large $i$ and $s\in\left(\f{\ep K_1 r_i}{\de_i},\f{K_2 r_i}{\de_i}\right)$, we have
\begin{equation}\aligned
\int_{M\cap\{\f{\ep K_1 r_i}{\de_i}<\mathcal{R}\le s\}}\mathcal{R}\na F_i\cdot\na\mathcal{R}\le&\int_{M\cap\{\f{\ep K_1 r_i}{\de_i}<\mathcal{R}\le s\}}\mathcal{R}\left(\overline{\na} F_i\cdot\overline{\na}\mathcal{R}+|\overline{\na} F_i|\cdot\left|\lan\overline{\na}\mathcal{R},\nu\ran\right|\right)\\
\le&\int_{M\cap\{\f{\ep K_1 r_i}{\de_i}<\mathcal{R}\le s\}}\left(C_2\ep+C_2\left|\lan\overline{\na}\mathcal{R},\nu\ran\right|\right)\\
\le&C_3\ep s^n+C_2\int_{M\cap\{\f{\ep K_1 r_i}{\de_i}<\mathcal{R}\le s\}}\left|\lan\overline{\na}\mathcal{R},\nu\ran\right|
\endaligned
\end{equation}
for some constants $C_2,C_3>1$, where the second inequality above has used \eqref{dfdR2} and \eqref{Rdf}. By the Cauchy inequality we get
\begin{equation}\aligned
&\limsup_{i\rightarrow\infty}\int_{\f{K_1 r_i}{\de_i}}^{\f{K_2r_i}{\de_i}}\left(\f1{s^{n+1}}\int_{M\cap\{\f{\ep K_1 r_i}{\de_i}<\mathcal{R}\le s\}}\mathcal{R}\na F_i\cdot\na\mathcal{R}\right)ds\\
\le&\limsup_{i\rightarrow\infty}\int_{\f{K_1 r_i}{\de_i}}^{\f{K_2r_i}{\de_i}}\left(\f{C_3\ep}s+\f {C_2}{s^{n+1}}\left(\int_{M\cap\{\f{\ep K_1 r_i}{\de_i}<\mathcal{R}\le s\}}\f{\lan\overline{\na}\mathcal{R},\nu\ran^2}{\mathcal{R}^n}\int_{M\cap\{\f{\ep K_1 r_i}{\de_i}<\mathcal{R}\le s\}}\mathcal{R}^n\right)^{\f12}\right)ds\\
\le&C_3\ep\log\f{K_2}{K_1}+C_4\limsup_{i\rightarrow\infty}\left(\int_{\f{K_1 r_i}{\de_i}}^{\f{K_2r_i}{\de_i}}\f1s ds\left(\int_{M\cap\{\f{\ep K_1 r_i}{\de_i}<\mathcal{R}\le \f{K_2r_i}{\de_i}\}}\f{\lan\overline{\na}\mathcal{R},\nu\ran^2}{\mathcal{R}^n}\right)^{\f12}\right)\\
\le&C_3\ep\log\f{K_2}{K_1}+C_4\log\f{K_2}{K_1}\limsup_{i\rightarrow\infty}\left(\int_{M\cap\{\f{\ep K_1 r_i}{\de_i}<\mathcal{R}\le \f{K_2r_i}{\de_i}\}}\f{\lan\overline{\na}\mathcal{R},\nu\ran^2}{\mathcal{R}^n}\right)^{\f12}.
\endaligned
\end{equation}
where $C_4$ is a constant.
Note $F_i$ is uniformly bounded for all $i$, then by Lemma \ref{limfdR} and \eqref{deiridRnu} we obtain
\begin{equation}\aligned
&\limsup_{i\rightarrow\infty}\left|\left(\f{\de_i}{K_2r_i}\right)^{n}\int_{M\cap\{\mathcal{R}\le \f{K_2r_i}{\de_i}\}}F_i|\overline{\na}\mathcal{R}|^2-\left(\f{\de_i}{K_1r_i}\right)^{n}\int_{M\cap\{\mathcal{R}\le \f{K_1 r_i}{\de_i}\}}F_i|\overline{\na}\mathcal{R}|^2\right|\\
\le&C_3\ep\log\f{K_2}{K_1}+C_4\limsup_{i\rightarrow\infty}\left(\ep^n\sup_{B_{\f{K_2r_i}{\de_i}}}|F_i|\right)
\le C_3\ep\log\f{K_2}{K_1}+C_5\ep^n
\endaligned
\end{equation}
for some constant $C_5$.
For any $\de\in(0,1)$, together with \eqref{Rr} we have
\begin{equation}\aligned
&\left|\f1{K_2^{n}}\int_{T\cap\left(\mathcal{B}_{K_2}\setminus\mathcal{B}_{\de K_1}\right)}f-\f1{K_1^{n}}\int_{T\cap\left(\mathcal{B}_{K_1}\setminus\mathcal{B}_{\de K_1}\right)}f\right|\\
=&\lim_{i\rightarrow\infty}\left|\left(\f{\de_i}{K_2r_i}\right)^{n}\int_{M\cap\{\f{\de K_1r_i}{\de_i}\le\mathcal{R}\le \f{K_2r_i}{\de_i}\}}F_i|\overline{\na}\mathcal{R}|^2-\left(\f{\de_i}{K_1r_i}\right)^{n}\int_{M\cap\{\f{\de K_1r_i}{\de_i}\le\mathcal{R}\le \f{K_1 r_i}{\de_i}\}}F_i|\overline{\na}\mathcal{R}|^2\right|\\
\le&\limsup_{i\rightarrow\infty}\left|\left(\f{\de_i}{K_2r_i}\right)^{n}\int_{M\cap\{\mathcal{R}\le \f{K_2r_i}{\de_i}\}}F_i|\overline{\na}\mathcal{R}|^2-\left(\f{\de_i}{K_1r_i}\right)^{n}\int_{M\cap\{\mathcal{R}\le \f{K_1 r_i}{\de_i}\}}F_i|\overline{\na}\mathcal{R}|^2\right|\\
&+\limsup_{i\rightarrow\infty}\left|\left(\f{\de_i}{K_2r_i}\right)^{n}\int_{M\cap\{\mathcal{R}\le \f{\de K_1r_i}{\de_i}\}}F_i|\overline{\na}\mathcal{R}|^2-\left(\f{\de_i}{K_1r_i}\right)^{n}\int_{M\cap\{\mathcal{R}\le \f{\de K_1 r_i}{\de_i}\}}F_i|\overline{\na}\mathcal{R}|^2\right|\\
\le& C_3\ep\log\f{K_2}{K_1}+C_5\ep^n+C_5\left(\f1{K_1^n}-\f1{K_2^n}\right)\limsup_{i\rightarrow\infty}\left(\f{\de_i^n}{r_i^n}\int_{M\cap\{\mathcal{R}\le \f{\de K_1 r_i}{\de_i}\}}1d\mu\right).
\endaligned
\end{equation}
Letting $\de\rightarrow0$ and $\ep\rightarrow0$ implies
\begin{equation}\aligned
\f1{K_2^{n}}\int_{T\cap\mathcal{B}_{K_2}}f=\f1{K_1^{n}}\int_{T\cap\mathcal{B}_{K_1}}f.
\endaligned
\end{equation}
By the argument in the proof of Theorem 19.3 in \cite{S}, the above
equality means that $T$ is a cone in $N_{\infty}$ up to a set of
measure zero, as $f$ is an arbitrary homogeneous function. In fact, by the coarea formula the above equality becomes
\begin{equation}\aligned
K_1^n\int_0^{K_2}\left(\int_{T\cap\p\mathcal{B}_{s}}f\right)ds=K_2^{n}\int_0^{K_1}\left(\int_{T\cap\p\mathcal{B}_{s}}f\right)ds.
\endaligned
\end{equation}
Differentiating w.r.t. $K_2$ and $K_1$  implies
\begin{equation}\aligned
\f1{K_2^{n-1}}\int_{T\cap\p\mathcal{B}_{K_2}}f=\f{1}{K_1^{n-1}}\int_{T\cap\p\mathcal{B}_{K_1}}f.
\endaligned
\end{equation}
Since $N_\infty=CX$ is a cone and any point in it can be represented by $(\r,\th)$ for some $\th\in X$, then we define
$\f1{r}T$ by $\{(\f\r r,\th)\in N_\infty|\ (\r,\th)\in T\}.$ So
\begin{equation}\aligned
\int_{\f1{K_2}T\cap\p\mathcal{B}_{1}}f=\int_{\f1{K_1}T\cap\p\mathcal{B}_{1}}f.
\endaligned
\end{equation}
Hence $\f1{K_2}T=\f1{K_1}T$ up to a set of measure zero, namely, $T$ is a cone, say, $CY$, where $Y\in \p\mathcal{B}_1(o)$ is an $(n-1)$-dimensional Hausdorff set. By \eqref{MEleast}, we know $\mathcal{H}^n(CY)>0$, which implies $\mathcal{H}^{n-1}(Y)>0$.
\end{proof}

\begin{remark}
By a simple modification, Lemma \ref{limfdR} and Lemma \ref{coneCY} also apply to minimal submanifolds of higher codimension with Euclidean volume growth in $N$.
\end{remark}
Without loss of generality, we assume that $\ep_iM$ converges to the cone $CY$ in the current sense. Let $\mathcal{X}\left(\mathcal{B}_{\f 2\ep}\setminus \mathcal{B}_{\ep}\right)$ be the set containing all smooth differential vector fields with compact support in $\mathcal{B}_{\f 2\ep}\setminus \mathcal{B}_{\ep}$ as in section 4.
For any $\xi\in\mathcal{X}\left(\mathcal{B}_{\f 2\ep}\setminus \mathcal{B}_{\ep}\right)$ let
\begin{equation}\aligned
\ep_iM(\omega)=\int_{\ep_iM}\lan\xi\circ\Phi_i^{-1},\nu_i\ran_i d\mu_i,\qquad CY(\xi\circ\Phi_i)=\int_{T}\lan\xi,\nu_\infty\ran d\mu_\infty,
\endaligned
\end{equation}
where $d\mu_i$ and $d\mu_\infty$ are the volume elements of $\ep_iM$ and $CY$, and $\nu_i$ and $\nu_\infty$ are the unit normal vectors of $\ep_iM$ and $CY$.

For any sufficiently small fixed constant $\ep\in(0,1)$, $\ep_iM\bigcap\left(B^i_{\f 2\ep}\setminus B^i_{\ep}\right)$ converges to $CY\bigcap\left(\mathcal{B}_{\f 2\ep}\setminus \mathcal{B}_{\ep}\right)$ in the varifold sense. Then
\begin{equation}\aligned\label{epiMvari}
\lim_{i\rightarrow\infty}\ep_iM\llcorner\left(B^i_{\f 2\ep}\setminus B^i_{\ep}\right)(\omega\circ\Phi_i^{-1})=CY\llcorner\left(\mathcal{B}_{\f 2\ep}\setminus \mathcal{B}_{\ep}\right)(\omega)
\endaligned
\end{equation}
for any $\omega\in\mathcal{X}\left(\mathcal{B}_{\f 2\ep}\setminus \mathcal{B}_{\ep}\right)$.

Let
\begin{equation}\label{Ei}
E_i\triangleq\left\{x\in \ep_iM\bigcap\left(B^i_{\f 2\ep}\setminus B^i_{\ep}\right)\bigg|\ \left|\left\lan \bn^{i}\r_i(x),\nu_i\right\ran_i\right|\ge\ep, \ \r_i\ \mathrm{is}\ C^1\ \mathrm{at}\ x\right\}.
\end{equation}
Note that $|\r_i(x)-\r_i(y)|\le d_i(x,y)$ for any $x,y\in\ep_iN$, where $d_i$ is the distance function on $\ep_iN$. So $\r_i$ is $C^1$-function almost everywhere (outside a set of $n$-dimensional Hausdorff measure zero). This set of measure zero does not affect any of the integrals in this paper, so we can assume that $\r_i$ is $C^1$ in these integrals.
If $\r_\infty(x)=d_\infty(o,x)$ is the distance function on $N_\infty$, then $\lim_{i\rightarrow\infty}\r\circ\Phi_i=\r_\infty$ in $B^i_{\f 2\ep}\setminus B^i_{\ep}$. For any compact set $K\subset \mathcal{B}_{\f 2\ep}\setminus \mathcal{B}_{\ep}$ by \eqref{epiMvari} we have
\begin{equation}\aligned
0=\lim_{i\rightarrow\infty}\left(\ep_iM\llcorner\Phi_i(K)\right)\left(\f{\p}{\p\r_\infty}\circ\Phi_i^{-1}\right)
=\lim_{i\rightarrow\infty}\int_{\ep_iM\cap\Phi_i(K)}\left\lan\f{\p}{\p\r_\infty}\circ\Phi_i^{-1},\nu_i\right\ran_i d\mu_i,
\endaligned
\end{equation}
and
\begin{equation}\aligned
0=\lim_{i\rightarrow\infty}\int_{\ep_iM\cap\Phi_i(K)}\left|\f{\p}{\p\r_\infty}\circ\Phi_i^{-1}-\bn^{i}\r_i\right|_i d\mu_i.
\endaligned
\end{equation}
By
\begin{equation*}\aligned
\left|\int_{\ep_iM\cap\Phi_i(K)}\left\lan\bn^{i}\r_i,\nu_i\right\ran_i d\mu_i\right|\le&\left|\int_{\ep_iM\cap\Phi_i(K)}\left\lan\f{\p}{\p\r_\infty}\circ\Phi_i^{-1},\nu_i\right\ran_i d\mu_i\right|\\
&+\int_{\ep_iM\cap\Phi_i(K)}\left|\f{\p}{\p\r_\infty}\circ\Phi_i^{-1}-\bn^{i}\r_i\right|_i d\mu_i,
\endaligned
\end{equation*}
we obtain
\begin{equation}\aligned\label{000Drinui}
0=\lim_{i\rightarrow\infty}\int_{\ep_iM\cap\Phi_i(K)}\left\lan\bn^{i}\r_i,\nu_i\right\ran_i d\mu_i.
\endaligned
\end{equation}

We claim that for sufficiently small $\ep>0$ there exists $i_0=i_0(\ep)$ such that $i>i_0$ implies
\begin{equation}\label{MEi}\mathcal{H}^n(E_i)<\ep^{n+1}.\end{equation}
If not, one could find a constant $\ep_0>0$ and a sequence $\mathbb{N}\ni s_i\rightarrow\infty$ such that $\mathcal{H}^n(E_{s_i})\ge\ep_0^{n+1}$. Then without loss of generality, there is a subsequence $s_{i_j}\rightarrow\infty$ of $s_i$ such that $\widetilde{E}_{s_{i_j}}\subset E_{s_{i_j}}$ with $\mathcal{H}^n(\widetilde{E}_{s_{i_j}})\ge\f12\ep_0^{n+1}$ and $\left\lan \bn^{i}\r_i(x),\nu_i\right\ran_i\ge\ep_0$ on $\widetilde{E}_{s_{i_j}}$. So we get $\mathcal{H}^n\left(\Phi^{-1}_{s_{i_j}}\left(\widetilde{E}_{s_{i_j}}\right)\right)\ge\ep_0^{n+2}$ if $\ep_0$ is sufficiently small and $j$ is sufficiently large. Note that $\ep_{s_i}M\rightharpoonup CY$. Then there are a set $K_0\subset CY\cap\left(\mathcal{B}_{\f 2\ep}\setminus \mathcal{B}_{\ep}\right)$ and a subsequence $s_{i_{j_k}}\rightarrow\infty$ of $s_{i_j}$ such that $K_0\subset \Phi^{-1}_{s_{i_{j_k}}}\left(\widetilde{E}_{s_{i_{j_k}}}\right)$ and $\mathcal{H}^n(K_0)>0$. Denote the sequence $s_{i_{j_k}}$ by $s_k$ for convenience. By \eqref{000Drinui}, we obtain
\begin{equation}\aligned
0=&\lim_{k\rightarrow\infty}\int_{\ep_{s_k}M\cap\Phi_{s_k}(K_0)}\left\lan\bn^{s_k}\r_{s_k},\nu_{s_k}\right\ran_{s_k} d\mu_{s_k}\\
\ge&\lim_{k\rightarrow\infty}\int_{\ep_{s_k}M\cap\Phi_{s_k}(K_0)}\ep_0 d\mu_{s_k}=\int_{CY\cap K_0}\ep_0 d\mu_\infty=\ep_0\mathcal{H}^n(K_0)>0.
\endaligned
\end{equation}
This is a contradiction, and we get the inequality \eqref{MEi}.

Now we assume that $M$ is a stable minimal hypersurface in $N$. Then $\ep_iM$ is still a stable minimal hypersurface in $\ep_iN$. Let $A^i$ be the second fundamental form of $\ep_iM$ in $\ep_iN$,
and $Ric_{\ep_iN}$ the Ricci curvature of $\ep_iN$.
For any Lipschitz function $\phi$ with compact support in $\ep_iM$ we have from (\ref{SV})
\begin{equation}\aligned\label{stableM}
\int_{\ep_iM}\left(|A^i|^2+Ric_{\ep_iN}(\nu_i,\nu_i)\right)\phi^2\le\int_{\ep_iM}|\na^i\phi|^2,
\endaligned
\end{equation}
where $\na^i$ is the Levi-Civita connection of $\ep_iM$. 
Now we suppose that there exists some sufficiently large $r_0>0$ such that the non-radial Ricci curvature of $N$ satisfies
\begin{equation}\aligned\label{RicxiT}
\inf_{\p B_r}Ric\left(\xi^T,\xi^T\right)\ge\f{\k'}{r^2}|\xi^T|^2
\endaligned
\end{equation}
almost everywhere for all $r\ge r_0$ and $n\ge2$, where $\k'$ is a positive constant, and $\xi^T$ stands for the part that is tangential to the geodesic sphere
$\p B_r$ (at least away from the cut locus of the center), of a tangent vector $\xi$ of $N$ at the considered point. Then
$$\inf_{\p B_s^i}Ric_{\ep_iN}\left(\e^T,\e^T\right)\ge\f{\k'}{r^2}|\e^T|_i^2>0$$
for all $s\ge \sqrt{\ep_i}r_0$ and $n\ge2$, where $\e$ is a local vector field on $\ep_iN$, $\e^T=\e-\left\lan \e,\bn^i\r_i\right\ran_i\bn^i\r_i$ if $\bn^i\r_i$ is well-defined.
Using conditions C1) and C3) which are both scaling invariant, we obtain
\begin{equation}\aligned\label{Ric}
Ric_{\ep_iN}(\nu_i,\nu_i)\ge& Ric_{\ep_iN}(\nu_i^T,\nu_i^T)+2\left\lan \nu_i,\bn^i\r_i\right\ran_i Ric_{\ep_iN}(\nu_i^T,\bn^i\r_i)\\
\ge& Ric_{\ep_iN}(\nu_i^T,\nu_i^T)-c'\left\lan \nu_i,\bn^i\r_i\right\ran_i\r_i^{-2}
\endaligned\end{equation}
for some absolute constant $c'>0$.
Let $\phi$ be the Lipschitz function on $\ep_iN$ defined by
$$\phi(x)=\left(\r_i(x)\right)^{\f{2-n}2}\sin\left(\pi\f{\log\r_i(x)}{\log\ep}\right)$$
in $B_1^i\setminus B_{\ep}^i$ and $\phi=0$ in other places.
Here $\ep$ is a small positive constant less than $\min\{\f12,\f{\k'}{2c'}\}$, which implies $\k'(1-\ep^2)-c'\ep\ge\f{\k'}4$.
So from (\ref{Ei}), (\ref{MEi}) and (\ref{Ric})
\begin{equation}\aligned
&\int_{\ep_iM}Ric_{\ep_iN}(\nu_i,\nu_i)\phi^2d\mu_i\\
\ge&\int_{(\ep_iM\setminus E_i)\cap(B_{1}^i\setminus B_{\ep}^i)}\left(\f{\k'}{\r_i^2}\left|\nu_i^T\right|_i^2-\f{c'}{\r_i^2}\left\lan \nu_i,\bn^i\r_i\right\ran_i\right)\sin^2\left(\pi\f{\log\r_i}{\log\ep}\right)\r_i^{2-n}d\mu_i\\
\ge&\left(\k'(1-\ep^2)-c'\ep\right)\int_{(\ep_iM\setminus E_i)\cap(B_{1}^i\setminus B_{\ep}^i)}\sin^2\left(\pi\f{\log\r_i}{\log\ep}\right)\r_i^{-n}d\mu_i\\
\ge&\left(\k'(1-\ep^2)-c'\ep\right)\left(\int_{\ep_iM\cap(B_{1}^i\setminus B_{\ep}^i)}\sin^2\left(\pi\f{\log\r_i}{\log\ep}\right)\r_i^{-n}d\mu_i-\ep^{-n}\mathcal{H}^n(E_i)\right)\\
\ge&\left(\k'(1-\ep^2)-c'\ep\right)\int_{\ep_iM\cap(B_{1}^i\setminus B_{\ep}^i)}\sin^2\left(\pi\f{\log\r_i}{\log\ep}\right)\r_i^{-n}d\mu_i-\k'\ep(1-\ep^2)
\endaligned
\end{equation}
for sufficiently large $i$. Substituting this into  \eqref{stableM} yields
\begin{equation}\aligned
&\left(\k'(1-\ep^2)-c'\ep\right)\int_{\ep_iM\cap(B^i_{1}\setminus B^i_{\ep})}\sin^2\left(\pi\f{\log\r_i}{\log\ep}\right)\r_i^{-n}d\mu_i-\k'\ep(1-\ep^2)\\
\le&\int_{\ep_iM}Ric_{\ep_iN}(\nu_i,\nu_i)\phi^2\le\int_{\ep_iM}|\overline{\na}^i\phi|_i^2\\
\le&\int_{\ep_iM\cap(B_1^i\setminus B^i_{\ep})}\left(\f{2-n}2\sin\left(\pi\f{\log\r_i}{\log\ep}\right)+\f{\pi}{\log\ep}\cos\left(\pi\f{\log\r_i}{\log\ep}\right)\right)^2\r_i^{-n}d\mu_i.
\endaligned
\end{equation}
Due to Lemma \ref{coneCY}, we let $i\rightarrow\infty$, and get
\begin{equation}\aligned
&\left(\k'(1-\ep^2)-c'\ep\right)\int_{CY\cap(\mathcal{B}_{1}\setminus \mathcal{B}_{\ep})}\sin^2\left(\pi\f{\log\r_\infty}{\log\ep}\right)\r_\infty^{-n}d\mu_\infty-\k'\ep(1-\ep^2)\\
\le&\int_{CY\cap(\mathcal{B}_1\setminus \mathcal{B}_{\ep})}\left(\f{2-n}2\sin\left(\pi\f{\log\r_\infty}{\log\ep}\right)+\f{\pi}{\log\ep}\cos\left(\pi\f{\log\r_\infty}{\log\ep}\right)\right)^2
\r_\infty^{-n}d\mu_\infty.
\endaligned
\end{equation}
Since
\begin{equation}\aligned
\int_{CY\cap(\mathcal{B}_{1}\setminus \mathcal{B}_{\ep})}\sin^2\left(\pi\f{\log\r_\infty}{\log\ep}\right)\r_\infty^{-n}d\mu_\infty=&\mathcal{H}^{n-1}(Y)\int_\ep^1\sin^2\left(\pi\f{\log s}{\log\ep}\right)\f1sds\\
=&\left(\log\f1\ep\right) \mathcal{H}^{n-1}(Y)\int_0^1\sin^2(\pi t)dt,
\endaligned
\end{equation}
and $\mathcal{H}^{n-1}(Y)>0$, then
\begin{equation}\aligned
&\left(\k'(1-\ep^2)-c'\ep\right)\left(\log\f1\ep\right) \mathcal{H}^{n-1}(Y)\int_0^1\sin^2(\pi t)dt-\k'\ep(1-\ep^2)\\
\le&\mathcal{H}^{n-1}(Y)\int_{\ep}^1\left(\f{2-n}2\sin\left(\pi\f{\log s}{\log\ep}\right)+\f{\pi}{\log\ep}\cos\left(\pi\f{\log s}{\log\ep}\right)\right)^2
\f1sds\\
=&\left(\log\f1\ep\right)\mathcal{H}^{n-1}(Y)\int_0^1\left(\f{2-n}2\sin(\pi t)+\f{\pi}{\log\ep}\cos(\pi t)\right)^2dt\\
=&\left(\log\f1\ep\right)\mathcal{H}^{n-1}(Y)\left(\f{(n-2)^2}4+\f{\pi^2}{(\log\ep)^2}\right)\int_0^1\sin^2(\pi t)dt,
\endaligned
\end{equation}
which implies
$$\k'\le\f{(n-2)^2}4.$$
Finally, we obtain the following results.
\begin{theorem}\label{N-M-E}
Let $N$ be an $(n+1)$-dimensional complete Riemannian manifold
satisfying conditions C1), C2) and C3), and with non-radial Ricci
curvature $\inf_{\p B_r}Ric\left(\xi^T,\xi^T\right)\ge\k'
r^{-2}$ almost everywhere for a constant $\k'$ and sufficiently large $r>0$, where $\xi^T$ stands for the part that is tangential to the geodesic sphere
$\p B_r$ (at least away from the cut locus of the center), of a tangent vector $\xi$ of $N$ at the considered point. If
$\k'>\f{(n-2)^2}4$, then $N$ admits no complete stable minimal
hypersurface with  at most Euclidean volume growth.
\end{theorem}

It is well known that area-minimizing hypersurfaces have Euclidean volume growth automatically.
Let $M$ be an $n$-dimensional area-minimizing hypersurface in $N$. Then the
$s$-dimensional Hausdorff measure of the singular set of $S$ is $H^{s}(\mathrm{Sing}\,M)=0$ for all $s>n-7$ (see \cite{S} for example).
We readily check that  Lemmas \ref{limfdR} and  \ref{coneCY} also hold for $M$. Namely,
there is a sequence $\{\ep_i\}$ converging to zero such that
$\ep_iN=(N,\ep_i \bar{g},p)$ converges to a metric cone
$(N_{\infty},d_{\infty})$, and $\ep_{i}M$ converges to the  cone $C\mathcal{Y}=\R^+\times_{\r}\mathcal{Y}$ in $N_{\infty}$, where $\mathcal{Y}\in \p\mathcal{B}_1(o)$ is an $(n-1)$-dimensional Hausdorff set.
\begin{corollary}\label{N-E-M}
Let $N$ be an $(n+1)$-dimensional  complete Riemannian manifold
satisfying conditions C1), C2) and C3), and with non-radial Ricci curvature $\inf_{\p B_r}Ric\left(\xi^T,\xi^T\right)\ge\k' r^{-2}$ for a constant $\k'$ and sufficiently large $r>0$, where $\xi$ is a local vector field on $N$ with $|\xi^T|=1$ defined in \eqref{RicxiT}. If $\k'>\f{(n-2)^2}4$, then $N$ admits no complete area-minimizing hypersurface.
\end{corollary}

\begin{remark}
$\k=\f{2}{n}\sqrt{n-1}$ in Remark \ref{MCSk} is equivalent to
$$Ric_{MCS_\k}\left(\xi^T,\xi^T\right)=\f{(n-2)^2}{4(\r+\f1\k-\r_0)^2}\qquad for\ all\ \r\ge\r_0,$$
where $\xi^T=\xi-\left\lan\xi,\f{\p}{\p\r}\right\ran\f{\p}{\p\r}$, $\left|\xi^T\right|=1$ and $\r_0\in(1,\f1\k)$ is a constant. Hence the constant $\k'$ in Theorem \ref{N-M-E} and Corollary \ref{N-E-M} is optimal.
\end{remark}

\bibliographystyle{amsplain}

\begin{thebibliography}{10}

\bibitem{Al} F. J. Almgren, Jr., Some interior regularity theorems for minimal surfaces and an extension of Bernstein's theorem, Ann. of Math. {\bf 85} (1966), 277-292.

\bibitem{AK} L. Ambrosio, B. Kirchheim, Currents in metric spaces, Acta Math. {\bf 185(1)} (2000), 1-80.

\bibitem{An} M. Anderson, On area-minimizing hypersurfaces in manifolds of nonnegative curvature, Indiana Univ. Math. J. {\bf 32} (1983), 745-760.



\bibitem{BDG} E. Bombieri, E. De Giorgi and E. Giusti, Minimal cones and the Bernstein problem, Invent. Math. {\bf 7} 1969, 243-268.

\bibitem{ChC} J. Cheeger and T. H. Colding, Lower bound on Ricci curvature and almost rigidity of warped product, Ann Math. {\bf 144(1)} (1996)189-237.

\bibitem{CG} J. Cheeger and D. Gromoll, The splitting theorem for manifolds of nonnegative Ricci curvature, J. Differential Geom. {\bf 6} (1971), 119-128.

\bibitem{CGT} J. Cheeger, M. Gromov, and M. Taylor, Finite propagation speed, kernel estimates for functions of the Laplace operator, and the geometry of complete Riemannian manifolds, J. Differential Geom. {\bf 17} (1982), 15-53.

\bibitem{CM1} T. H. Colding and W. P. Minicozzi II, Large scale Behavior of Kernels of Schr$\mathrm{\ddot{o}}$dinger Operators, Amer. J. Math. {\bf 119 (6)} (1997), 1355-1398.

\bibitem{DG} E. De Giorgi, Una estensione del teorema di Bernstein, Ann. Sc. Norm. Sup. Pisa {\bf 19} (1965), 79-85.

\bibitem{D} Q. Ding, The inverse mean curvature flow in rotationally symmetric spaces, Chin. Ann. Math. B {\bf 32B(1)} (2011), 27-44.




\bibitem{FF} H. Federer, W.H. Fleming, Normal and integral currents, Ann. Math. {\bf 72(2)} (1960), 458-520.

\bibitem{FS} D. Fischer-Colbrie, R. Schoen, The structure of complete stable minimal surfaces in 3-manifolds of nonnegative scalar curvature, Commun. Pure Appl. Math. {\bf 33} (1980), 199-211.

\bibitem{F} W. Fleming, On the oriented Plateau problem, Rend Circ. Mat. Palermo {\bf 11} (1962), 1-22.

\bibitem{GHL} S. Gallot, D. Hulin and J. Lafontaine, Riemannian geometry, Third edition, Universitext, Springer-Verlag, Berlin, 2004.

\bibitem{Gi} E. Giusti, Minimal surfaces and functions of bounded variation, Birkh$\mathrm{\ddot{a}}$user Boston, Inc., 1984.

\bibitem{GW} R. E. Greene, H. Wu, Lipschitz convergence of Riemannian manifolds, Pac. J. Math. {\bf 131} (1988), 119-141.

\bibitem{GLP} M. Gromov, J. Lafontaine and P. Pansu, Structures
  metriques pour les variete riemannienes, Cedic-Fernand Nathan, Paris
  (1981).

\bibitem{JK} J. Jost and H. Karcher, Geometrische Methoden zur Gewinnung
  von a-priori-Schranken f\"ur harmonische Abbildungen,
  manuscr.math. {\bf 40} (1982), 27-77.

\bibitem{L} G. R. Lawlor, A sufficient criterion for a cone to be
  area-minimizing, Mem. AMS {\bf 446} (1991).

\bibitem{Pl} P. Li, Harmonic Functions and Applications to Complete Manifolds, University of California, Irvine, 2004, preprint.


\bibitem{MSY} N. Mok, Y.T. Siu and S.T. Yau, The Poincar$\mathrm{\acute{e}}$-Lelong equation on complete K$\mathrm{\ddot{a}}$hler manifolds, Compositio Mathematica {\bf 44} (1981), 183-218.

\bibitem{Mf} F. Morgan, Area-minimizing surfaces in cones, Communications in analysis and geometry {\bf 10(5)} (2002), 971-983.

\bibitem{N} P. Nabonnand, Sur les vari$\mathrm{\acute{e}}$t$\mathrm{\acute{e}}$s Riemanniennes compl$\mathrm{\grave{e}}$tes $\mathrm{\grave{a}}$ courbure de Ricci positive, R. Acad. Sci. Paris {\bf 291} (1980), 591-593.

\bibitem{P} G. Perelman, A complete Riemannian manifold of positive Ricci curvature with Euclidean volume growth and nonunique asymptotic cone, Comparison geometry 165-166, Cambridge Univ. Press Cambridge, 1977.

\bibitem{Ps} S. Peters, Convergence of Riemannian manifolds, Compositio
  Math. {\bf 62} (1987), 3-16.

\bibitem{SSY} R. Schoen, L. Simon, and S. T. Yau, Curvature estimates for minimal hypersurfaces, Acta Math. {\bf 134} (1975), 275-288.

\bibitem{SZ} Y. Shen and S. Zhu, Rigidity of stable minimal hypersurfaces, Math. Ann. {\bf 309} (1997), no. 1, 107-116.

\bibitem{S} L. Simon, Lectures on Geometric Measure Theory, Proceedings of the center for mathematical analysis Australian national university, Vol. 3, 1983.

\bibitem{Si} J. Simons, Minimal varieties in Riemannian manifolds, Ann. Math. {\bf 88} (1968), 62-105.

\bibitem{Sp} J. Spruck, Interior gradient estimates and existence theorems for constant mean curvature graphs in $M^n\times\R$, Pure Appl. Math. Q. {\bf 3} (2007), no. 3, Special Issue: In honor of Leon Simon. Part 2, 785-800.


\bibitem{X} Y. L. Xin, Minimal Submanifolds and Related Topics, World Scientific Publ., (2003).


\end{thebibliography}

\end{document}